\documentclass[10pt]{amsart}

\usepackage{amssymb,amsmath,amsthm,amscd}
\usepackage{dsfont, bbm}
\usepackage{tikz}

\advance\oddsidemargin by -1cm
\advance\evensidemargin by -1cm 
\newtheorem{theorem}{Theorem}
 
\newtheorem{proposition}{Proposition}
\newtheorem{corollary}{Corollary}
\newtheorem{lemma}{Lemma}
\newtheorem{definition}{Definition}

\theoremstyle{definition}
\newtheorem{remark}{Remark}

\newcommand{\bdm}{\begin{displaymath}}
\newcommand{\edm}{\end{displaymath}}
\newcommand{\bq}{\begin{equation}}
\newcommand{\eq}{\end{equation}}
\newcommand{\bqn}{\begin{equation*}}
\newcommand{\eqn}{\end{equation*}}

\newcommand{\rn}{\mathbb{R}^n}

\newcommand{\eps}{\varepsilon}

\newcommand{\Reg}{\mathrm{Reg}\,}

\newcommand{\norm}[1]{\left\| #1 \right\|}
\newcommand{\mklm}[1]{\left\{ #1 \right\}}
\newcommand{\eklm}[1]{\left\langle  #1 \right\rangle}

\renewcommand{\d}{\,d}
\newcommand{\N}{{\mathbb N}}

\newcommand{\C}{{\mathbb C}}
\newcommand{\R}{{\mathbb R}}
\newcommand{\K}{{\mathbb K}}

\newcommand{\B}{{\mathcal B}}
\newcommand{\D}{{\mathcal D}}
\newcommand{\E}{{\mathcal E}}
\newcommand{\F}{{\mathcal F}}

\newcommand{\M}{{\mathcal M}}

\newcommand{\X}{{\mathbb X}}

\renewcommand{\epsilon}{\varepsilon}
\renewcommand{\phi}{\varphi}
\renewcommand{\rho}{\varrho}
\newcommand{\1}{{ \bf  1}}

\newcommand{\Cinft}{{\rm C^{\infty}}}
\newcommand{\CT}{{\rm C^{\infty}_c}}

\renewcommand{\L}{{\rm L}}

\newcommand{\Ncal}{{\mathcal N}}

\renewcommand{\S}{{\mathcal S}}
\newcommand{\G}{{\mathcal G}}

\newcommand{\Syms}{{\rm S^{-\infty}}}

\newcommand{\GL}{\mathrm{GL}}

\newcommand{\ltwo}{L^2(\Gamma \backslash G)}
\newcommand{\gmg}{\Gamma \backslash G}
\newcommand{\g}{{\bf \mathfrak g}}
\renewcommand{\k}{{\bf \mathfrak k}}

\newcommand{\p}{{\bf \mathfrak p}}

\newcommand{\U}{{\mathfrak U}}

\newcommand{\Ad}{\mathrm{Ad}\,}
\newcommand{\ad}{\mathrm{ad}\,}

\newcommand{\id}{\mathrm{id}\,}

\renewcommand{\det}{\mathrm{det}\,}

\renewcommand{\Re}{\mathrm{Re}\,}
\newcommand{\vol}{\text{vol}\,}

\newcommand{\Crit}{\mathrm{Crit}}

\DeclareMathOperator{\supp}{supp}

\DeclareMathOperator{\tr}{tr}
\DeclareMathOperator{\gd}{\partial}

\newcommand{\e}[1]{\,{\mathrm e}^{#1}\,}
\newcommand{\dbar}{{\,\raisebox{-.1ex}{\={}}\!\!\!\!d}}

\begin{document}

\author{Octavio Paniagua-Taboada and Pablo Ramacher}
\title{Equivariant  heat asymptotics  on spaces of automorphic forms}
\address{Octavio Paniagua-Taboada and Pablo Ramacher, Fachbereich Mathematik und Informatik, Philipps-Universit\"at Marburg,  
Hans-Meerwein-Str., 35032 Marburg, Germany}
\email{paniagua@mathematik.uni-marburg.de, ramacher@mathematik.uni-marburg.de}
\thanks{ The authors wish to thank Roberto Miatello for  his encouragement, and  many stimulating conversations. This work was financed by the DFG-grant RA 1370/2-1.}
\date{\today}
\begin{abstract}
Let $G$ be a connected, real, semisimple Lie group with finite center, and  $K$ a maximal compact subgroup of $G$. 
In this paper, we derive $K$-equivariant asymptotics  for heat traces with remainder estimates on compact Riemannian manifolds carrying a transitive and isometric $G$-action. In particular, we compute the leading coefficient in the Minakshishundaram-Pleijel expansion of the heat trace for Bochner-Laplace operators on homogeneous vector bundles over compact locally symmetric spaces of arbitrary rank. 
\end{abstract}

\maketitle

\tableofcontents

\section{Introduction}

Let $G$ be a connected, real, semisimple Lie group with finite center,  acting isometrically and transitively on a  compact, $n$-dimensional, real-analytic  Riemannian manifold $M$. Let further $K$ be a maximal compact subgroup of $G$. In this paper, we derive $K$-equivariant asymptotics for traces of heat semigroups associated to strongly elliptic operators on $M$ with remainder estimates. In particular, if $M=\Gamma \backslash G$, where $\Gamma$ is a discrete, torsion-free, uniform subgroup of $G$, we compute the leading coefficient in the Minakshishundaram-Pleijel expansion of the  heat trace of   Bochner-Laplace operators on homogeneous vector  bundles over compact, locally symmetric spaces of arbitrary rank, together with an estimate for the remainder.

The study of the asymptotic behavior  of heat semigroups and their kernels has a long history. One of the pioneering works in this direction was the derivation of an asymptotic expansion for the fundamental solution of the heat equation on a compact manifold  by Minakshisundaram and Pleijel  \cite{minakshisundaram-pleijel49}. The first three coefficients in this expansion were computed by McKean and Singer  \cite{mckean-singer67} in terms of geometric quantities, yielding corresponding expansions of heat traces. This culminated 
in a heat theoretic proof of the index theorem by Atiyah, Bott and Patodi \cite{atiyah-bott-patodi73}. 
 In the case of Riemannian symmetric spaces, an explicit expression for the fundamental solution of the heat equation was given by Gangolli  \cite{gangolli68} using Harish-Chandra's Plancherel theorem. Later, Donelly \cite{donnelly79} generalized the constructions in \cite{minakshisundaram-pleijel49} and \cite{BGM71}  to Riemannian manifolds admitting a properly discontinuous group of isometries with compact quotient. 
Following these developments,   Miatello \cite{miatello80},  and DeGeorge and Wallach \cite{degeorge-wallach79} established  asymptotic expansions for heat traces of Bochner-Laplace operators on homogeneous vector bundles over compact, locally symmetric spaces of rank one. 
Holomorphic semigroups generated by strongly elliptic operators on Lie groups have been studied sytematically by Langlands \cite{langlands60}, and Robinson and ter Elst \cite{robinson}, \cite{ter_elst-robinson}, giving lower and upper bounds for their kernels. For further references, see also \cite{davies} and  \cite{berline-getzler-vergne}.

To illustrate our results, let $(\pi,\L^2(M))$ be the regular representation of $G$ on the Hilbert space of square integrable functions on $M$ with respect to an invariant density, and  $f_t$ the group kernel of a strongly elliptic operator $\Omega$ of order $q$ associated to the representation $\pi$, where  $t>0$.  The corresponding heat operator is then given by $e^{-t\overline{\Omega}}=\pi(f_t)$, and characterized in Theorem \ref{thm:2} as a pseudodifferential operator of order $-\infty$. Due to the compactness of $M$,  this implies that $\pi(f_t)$ is of trace class. Using this characterization, we  consider the decomposition 
\bqn 
\L^2(M)=\bigoplus_{\sigma \in \widehat K} \L^2(M)_\sigma
\eqn
of $\L^2(M)$ into $K$-isotypic components, and  derive asymptotics with remainder estimates for the trace 
$$
\tr (P_\sigma\circ \pi(f_t) \circ P_\sigma)
$$
of the restriction of $\pi(f_t)$ to the isotypic component $\L^2(M)_\sigma=P_\sigma (\L^2(M))$ as $t$ goes to zero, $P_\sigma $ being the corresponding projector, see Theorem \ref{thm:B}.  In order to do so, one has to describe the asymptotic behavior of certain oscillatory integrals, which has been determined before in \cite{ramacher10} while studying   the spectrum of an invariant elliptic operator. The difficulty here resides in the fact that, since the  critical sets of the corresponding phase functions are not smooth, a desingularization procedure is required in order to apply the method of the stationary phase in  a suitable resolution space.  In  case that $f_t$ has  an asymptotic expansion of the form
\bqn 
f_t(g)\sim \frac{1}{ t^{d/q}} e^{-b\big (\frac{d(g,e)^q}t\big )^{1/(q-1)}} \sum_{j=0}^\infty c_j(g) t^j,  \qquad b>0, 
\eqn
near the identity $e\in G$ with analytic  coefficients $c_j(g)$,
where $d=\dim G$, and  $d(g,e)$ denotes the distance of $g \in G$ from the identity with respect to the canonical left-invariant metric on $G$,  we show in Corollary \ref{cor:B} that 
\begin{align*}
\tr  (P_\sigma\circ \pi(f_t) \circ P_\sigma)&=\frac{d_{\sigma\otimes \sigma} [(\pi_\sigma\otimes \pi_\sigma)_{|\mathbb{H}}:\1] }{(2\pi)^{n-\kappa}t^{(n-\kappa)/q}} c_0(e) \, \widetilde {\vol} (\Xi/\K) + O(t^{-(n-\kappa-1)/q} (\log t)^{\Lambda -1}),
\end{align*}
where  $(\pi_\sigma, V_\sigma) \in \sigma$, and $\widetilde {\vol} (\Xi/\K)$ is given by local integrals over the zero level set $\Xi=\mathbb{J}^{-1}(0)$ of the momentum map $\mathbb{J}:T^\ast M \rightarrow (\k\oplus \k)^\ast$ of the underlying action of $\K=K\times K$  on $M$. In fact, $\widetilde {\vol} (\Xi/\K)$ represents a Gaussian volume of the symplectic quotient $\Xi/\K$. 
Further,  $\kappa$ denotes the dimension of a $K$-orbit of principal type, and $\mathbb{H}\subset \K$ a principal isotropy group, while $\Lambda$ is the maximal number of elements of a totally ordered subset of the set of $\K$-isotropy types. 

As our main application, we consider the case  $M=\Gamma \backslash G$, where  $\Gamma\subset G$ is a discret, co-compact subgroup. The previous results, combined  with Selberg's trace formula, then yield an asymptotic description of $L_\sigma f_t$ at the identity, where $L_\sigma$ denotes the projector onto the isotypic component $\L^2(G)_\sigma$ of the left-regular representation $(L,\L^2(G))$ of $G$, see Proposition \ref{prop:2}. Finally,  for torsion-free $\Gamma$, we are able to compute the first coefficient in the Minakshisundaram-Pleijel expansion, together with an estimate for the remainder,  of vector valued heat kernels on the compact locally symmetric space $\Gamma\backslash G/K$, generalizing part of the work in \cite{miatello80} and \cite{degeorge-wallach79} to   arbitrary rank. More precisely, let  $\Delta_\sigma$ be  the Bochner-Laplace operator on the homogeneous vector bundle $E_\sigma =\Gamma \backslash (G\times V_\sigma)/K\rightarrow \gmg/K$. Denote  by $\lambda_\sigma$ the Casimir eigenvalue of $K$ corresponding to $\sigma \in \widehat K$. Then, by Theorem \ref{thm:5},  
\begin{gather*}
\tr e^{-t \Delta_\sigma} =\frac{ e^{t\lambda_\sigma}  \int_{\mathbb{H}} \tr \pi_\sigma(kk_1^{-1}) \d k_1 \d k}{(2\pi)^{\dim G/K}t^{\frac{\dim G/K}{2}}}\widetilde {\vol} (\Xi/\K)+ O(e^{t \lambda_\sigma}t^{-(\dim G/K-1)/2} (\log t)^{\Lambda -1}),
\end{gather*}
where, again,   $\widetilde {\vol} (\Xi/\K)$ is given by a Gaussian volume of the symplectic quotient $\Xi/\K$. 

 This paper is organized as follows. The microlocal structure of general convolution operators with rapidly decaying group kernels on paracompact, smooth manifolds is described in  Section 2. In Section 3, the Langlands kernel of a semigroup generated by a strongly elliptic operator on $M$ is considered, and its equivariant heat trace is expressed in terms of oscillatory integrals. Since the occuring phase functions do have singular critical sets, the stationary phase principle  cannot immediately be applied to describe the asymptotic behavior of those integrals. Instead, we rely on  the results in   \cite{ramacher10}, where  resolution of singularities was used to partially resolve the singularities of the considered critical sets. This yields  short-time asymptotics with remainder estimates for equivariant heat traces in Section 4. Finally, in Section 5, we consider the particular case  $M=\Gamma \backslash G$, where $\Gamma$ denotes a uniform, torsion-free lattice in $G$, and apply our results to heat traces of  Bochner-Laplace operators on compact locally symmetric spaces of arbitrary rank.

\section{Convolution operators}

Let $G$ be a connected, real, semisimple Lie group with finite center, and Lie algebra $\g$.  Denote by  $\langle X,Y\rangle = \tr \, (\ad X\circ \ad Y)$ the Cartan-Killing form   on $\g$, and by $\theta$   a Cartan involution  of $\g$. Let 
\bq
\label{eq:cartan}
\g = \k\oplus\p
\eq
be  the Cartan decomposition of $\g$ into the eigenspaces of  $\theta$, corresponding to the eigenvalues  $+1$ and $-1$ , respectively. Put $\langle X,Y\rangle _\theta:=-\langle X,\theta Y\rangle $. Then $\langle \cdot,\cdot \rangle _\theta$ defines a left-invariant metric on $G$.  With respect to this metric, we define  $d(g,h)$ as  the geodesic distance between two points $g,h \in G$, and set $|g|=d(g,e)$, where $e$ is the identity element of $G$. Note that $d(g_1g,g_1h) = d(g,h)$ for all $g,g_1,h \in G$.  In contrast to the Killing form,  $\langle \cdot,\cdot \rangle _\theta$ is no longer $\Ad(G)$-invariant, but  still $\Ad(K)$-invariant, so that $d(gk,hk)=d(g,h)$ for all $g,h \in G$, and $k \in K$.   Indeed, one has the following 
\begin{proposition}\label{prop:Ad(k)inv}
The modified Killing form $\left<\cdot  ,\cdot \right>_\theta$ is $\Ad(K)$-invariant, which implies that the corresponding Riemannian distance $d$ on $G$ is right $K$-invariant. In particular, $|g|=|kgk^{-1}|$ for all $g \in G$ and $k \in K$. 
\end{proposition}
\begin{proof}
This seems to be a well-known fact, but  for  lack of  references, we include a proof here. 
 Thus, let us  first note that for $k \in K$, the morphisms $\Ad (k)$ and $\theta$ commute. Indeed,  the inclusions $[\k,\k]\subset \k$, $[\p,\p]\subset \k$, and $[\k,\p] \subset \p$, together with the relation $\Ad(e^X)= e^{\ad X}$, $X \in \g$, imply that $\Ad(K)  \, \k \subset \k$, $\Ad(K) \, \p \subset \p$. Hence, $\Ad(k) \theta X= \theta \Ad(k)  X$ for all $X \in \g$. 
But then
 $$\left<\Ad(k)X , \Ad(k)Y\right>_\theta = -\left<\Ad(k)X , \theta \Ad(k)Y\right> = -\left<\Ad(k)X ,  \Ad(k)\theta Y\right>= - \left<X , \theta Y\right>=  \left<X , Y\right>_\theta
 $$
 for all $X,Y \in \g$, $k \in K$, 
 showing  the $\Ad(K)$-invariance of $\left<\cdot ,\cdot \right>_\theta$. 
 
Next, we  show that the Riemannian distance $d$ is right $K$-invariant. For this purpose, recall that for a curve $c : [a,b] \to {\bf{X}}$ on a Riemannian manifold ${\bf{X}}$ with metric $\nu$,  the length of $c$ is given by
 \[
 L(c)= \int_a^b \sqrt{ \nu_{c(s)} (c'(s), c'(s)) }ds.
 \]
Let now  $c:[a,b]\rightarrow G$ be a curve in $G$ joining two points $g,h\in G$. 
We then assert that 
 \begin{equation}\label{eq0ad(k)}
 \frac{d}{dt} kc(t)k^{-1}|_{t=t_0} =  \left( dL_{kc(t_0)k^{-1}} \right)_e \Ad(k)\Big ( (dL_{c(t_0)^{-1}})_{c(t_0)} c'(t_0) \Big ), \qquad k \in G,
 \end{equation}
 where $L_g:G \to G$ corresponds to  left-translation by $g\in G$, and $(dL_g)_h: T_h G \to T_{gh} G$ is its differential at  $h \in G$. Indeed, if  $i_k: G \to G$ denotes the interior automorphism $h \mapsto khk^{-1}$, its differential at the identity $e$ is by definition $\Ad(k)=(d i_k)_e :  \g \to \g$. Furthermore, since $i_k = L_k \circ R_{k^{-1}} = R_{k^{-1}} \circ L_k$, we have the  identities
 \begin{equation*}\label{eqAd(k)}
 \Ad(k)= (dL_k)_{k^{-1}} \circ (dR_{k^{-1}})_e =  (dR_{k^{-1}})_k \circ (dL_k)_e.
 \end{equation*}
Similarly,  $L_{kc(t_0)k^{-1}}= L_k \circ L_{c(t_0)} \circ L_{k^{-1}}$ implies $(dL_{kc(t_0)k^{-1}})_e = (dL_k)_{c(t_0)k^{-1}} \circ (dL_{c(t_0})_{k^{-1}} \circ (dL_{k^{-1}})_e$. The left hand side of \eqref{eq0ad(k)} now reads
\[
\frac{d}{dt} kc(t)k^{-1}|_{t=t_0} = (dL_k)_{c(t_0)k^{-1}} \circ (dR_{k^{-1}})_{c(t_0)}(c'(t_0)),
\]
while the right hand side equals
\[
(dL_k)_{c(t_0)k^{-1}}\circ (dL_{c(t_0)})_{k^{-1}}  \circ (dL_{k^{-1}})_e \circ (dL_k)_{k^{-1}} \circ (dR_{k^{-1}})_{e} \circ (dL_{c(t_0)^{-1}})_{c(t_0)}(c'(t_0)) 
\]
\[
= (dL_k)_{c(t_0)k^{-1}}\circ (dR_{k^{-1}})_{c(t_0)} \circ (dL_{c(t_0)})_{e}  \circ (dL_{c(t_0)^{-1}})_{c(t_0)}(c'(t_0))
\]
\[
= (dL_k)_{c(t_0)k^{-1}}\circ (dR_{k^{-1}})_{c(t_0)}(c'(t_0)),
\]
prooving  \eqref{eq0ad(k)}. Write  $\eklm{X,X}_\theta=\norm{X}^2_\theta$. The $\Ad(K)$-invariance of $\eklm{\cdot,\cdot}_\theta$ then implies  
\begin{align*}
L(kck^{-1})&=\int_a^b \norm{(dL_{kc(s)^{-1}k^{-1}})_{kc(s)k^{-1}}\Big (\frac{d}{dt} kc(t)k^{-1}|_{t=s}\Big )}_\theta \d s \\ 
&=\int_a^b \norm{ \Ad(k) \left[ (dL_{c(t_0)^{-1}})_{c(t_0)} c'(t_0) \right]}_\theta \d s=\int_a^b \norm{  (dL_{c(t_0)^{-1}})_{c(t_0)} c'(t_0) }_\theta \d s=L(c)
\end{align*}
for arbitrary $k \in K$. 
Assume now that $c$ is a shortest geodesic. The last equality then shows that  $kck^{-1}$ is a shortest geodesic, too. Otherwise there would exist a geodesic $\tilde c$ joining $kgk^{-1}$ and $khk^{-1}$ with $L(\tilde c) < L(kck^{-1})$. But then $L(k^{-1}\tilde c k)<L(c)$, a contradiction. Therefore
\bqn 
d(g,h)=L(c) = L(kck^{-1})=d(kgk^{-1}, khk^{-1}) = d(gk^{-1}, hk^{-1})
\eqn
for all $g,h \in G$, $k \in K$, and the proposition follows.
\end{proof}

Let us consider next  a paracompact $\Cinft$-manifold $M$ of dimension $n$, and assume that $G$ acts on  $M$ in a smooth and transitive way. Let $\mathrm{C}(M)$ be the Banach space of continuous, bounded, complex valued functions on $M$, equipped with the supremum norm, and let  $(\pi,\mathrm{C}(M))$ be the corresponding continuous  regular representation of $G$ given by
\bqn 
\pi(g) \phi({p}) =\phi(g \cdot {p}), \qquad \phi \in \mathrm{C}(M), \quad g \in G, \quad p \in M.
\eqn
  The representation of the universal enveloping algebra $\U$ of the complexification $\g_\C$ of $\g$ on the space of differentiable vectors $\mathrm{C}(M)_\infty$ will be denoted by $d\pi$. We shall also consider the regular representation  of $G$ on $\Cinft(M)$ which, equipped with the topology of uniform convergence on compacta, becomes a Fr\'{e}chet space. This representation will be denoted by $\pi$ as well. Let $(L,\Cinft (G))$ and $(R,\Cinft (G))$ be the left, respectively right regular representation of $G$.    A function $f$ on $G$ is said to be of \emph{at most of exponential growth}, if there exists a $\kappa>0$ such that $|f(g)| \leq C e^{\kappa|g|}$ for some constant $C>0$, and all $g \in G$. Let $\d g$ be  a Haar measure on $G$. We then make the following
\begin{definition}
\label{def:1}
The space of rapidly decreasing functions on $G$, denoted by $\S(G)$, is given by all functions $f \in \Cinft(G)$ satisfying the following conditions:
\begin{itemize}
\item[i)] For every $\kappa \geq 0$, and $X \in \U$, there exists a constant $C > 0$ such that 
	$$|dL(X)f(g)| \leq C e^{-\kappa |g|} ;$$ 
\item[ii)] for every $\kappa \geq 0$, and $X \in \U$, one has $dL(X)f \in \L^1(G,e^{\kappa|g|}d_G)$.
\end{itemize}
\end{definition}
The space $\S(G)$ was first introduced in  \cite{ramacher06}, and motivated by the study of strongly elliptic operators, and the semigroups generated by them, see Section 3. Let us  now associate to every $f\in \S(G)$ and $\phi \in \mathrm{C}(M)$ the vector-valued integral $\int _{G} f(g)  
\pi(g) \phi \d_{G}(g)\in \mathrm{C}(M)$, yielding a  continuous linear operator
\bq
\label{eq:1}
\pi(f)=\int_G f(g) \pi(g) \d g
\eq
on $\mathrm{C}(M )$. Its restriction to $\CT(M)$ induces a continuous linear operator            
\begin{equation*}
\pi(f):\CT(M) \longrightarrow \mathrm{C}(M) \subset \D'(M),
\end{equation*}
with Schwartz kernel given by the distribution section  $\mathcal{K}_f \in  \D'(M \times M, {{\bf 1}} \boxtimes \Omega_{M})$, where $\Omega_M$ denotes the density bundle of $M$.   In what follows,  we shall show  that $\pi(f)$ is an operator with smooth kernel. As we shall see, the smoothness of the operators $\pi(f)$  is a direct consequence of the fact that $G$ acts transitively on $M$.

Thus, let $\mklm{(\widetilde W_\iota', \phi_{\iota})}_{\iota \in I}$ be a locally finite  atlas of $M$. By \cite{kobayashi-nomizuI}, page 273, there exists a locally finite refinement $\mklm{\widetilde W_\iota}_{\iota \in I}$ with the same index set such that $\overline{\widetilde W_\iota} \subset \widetilde W_\iota'$ for every $\iota \in I$. Assume that the 
$\overline{\widetilde W_\iota'}$ are compact, and let $\mklm{\alpha_\iota}_{\iota \in I}$ be a partition of unity subordinated to the atlas $\mklm{(\widetilde W_\iota, \phi_{\iota})}_{\iota \in I}$, meaning that
\begin{enumerate}
\item[(a)] the $\alpha_\iota$ are smooth functions, and $0 \leq \alpha_\iota \leq 1$;
\item[(b)] $\supp \alpha_\iota \subset \widetilde W_\iota$;
\item[(c)] $\sum_{\iota \in I} \alpha_\iota =1$.
\end{enumerate}
Let further $\mklm{\alpha'_\iota}_{\iota\in I}$ be another set of functions satisfying  condition (a), and in addition
\begin{enumerate}
\item[(b')] $\supp \alpha'_\iota \subset \widetilde W_\iota'$;
\item[(c')] $\alpha'_{\iota|\widetilde W_\iota}\equiv 1$. 
\end{enumerate}
  Consider now the localization  of $\pi(f)$ with respect to the latter atlas
 \bqn 
A_{f}^{\iota} u=[\pi(f)_{|\widetilde W_{\iota}}(u\circ
\phi_{\iota})]\circ \phi_{\iota}^{-1}, \qquad
u\in\CT(W_{{\iota}}), \,
W_{\iota}=\phi_{\iota}(\widetilde W_\iota)\subset \R^{n},
 \eqn 
corresponding to  the diagram
 \begin{displaymath}
\begin{CD} 
\CT(\widetilde W_{\iota})       @>{\pi(f)_{|\widetilde W_\iota}}>>   \Cinft(\widetilde W_\iota)               \\
@A {\phi_{\iota}^\ast}AA @AA {\phi_{\iota}^\ast}A\\
 \CT( W_{\iota})  @> {A_f^\iota}>>  \Cinft( W_\iota).       
\end{CD}
\end{displaymath}
Let $ {p} \in \widetilde W_{\iota}$.   Writing $\phi_{\iota}^{g}= \phi_{\iota}\circ g\circ \phi_{\iota}^{-1}$, and $x=\phi_\iota({p})= (x_{1},\dots,x_{n})\in {W}_\iota$ we obtain
\begin{equation*}
A_{f}^{\iota} u(x)=\int_{G}  f(g)[(u\circ \phi_{\iota}) \alpha'_\iota](g \cdot \phi_\iota^{-1}(x)) \, \d g=\int_{G}  f(g) c_{\iota}(x,g)(u\circ \phi_{\iota}^{g})(x)\d g,
\end{equation*}
where we put $c_{\iota}(x,g)=\alpha'_\iota(g \cdot \phi_\iota^{-1}(x))$. Next, define the functions
\bq
\label{eq:sym}
a_{f}^{\iota}(x,\xi)=e^{-ix\cdot \xi} \int_{G}e^{i\phi_{\iota}^{g}(x)\cdot\xi} c_{\iota}(x,g)f(g) \d g.
\eq
Since $f$ is rapidly falling, differentiation under the integral yields $a_{f}^{\iota}(x,\xi) \in\Cinft(W_{{\iota}} \times \R^{n})$. We can now state 

\begin{theorem}[Structure theorem]
\label{thm:2}
Let $M$ be a  paracompact $\Cinft$-manifold of dimension $n$,  and $G$ a connected, real, semisimple Lie group with finite center acting on $M$ in a smooth and transitive way. Let further  $f\in \S({G})$ be a rapidly decaying function on $G$.   Then the operator $\pi(f)$ is a pseudodifferential operator of class  $\L^{-\infty}(M)$, that is, it is locally of the form \footnote{Here and in what follows we  use the convention that, if not specified otherwise, integration is to be performed 
over whole Euclidean space.}
\begin{gather}
\label{eq:2}
  A^\iota_fu(x)= \int e ^{i x \cdot \xi} a_f^\iota(x,\xi)\hat u(\xi) \, \dbar\xi, \qquad u \in \CT(W_\iota),
\end{gather}
where the symbol $a_f^\iota(x,\xi)\in \mathrm{S}^{-\infty}(W_\iota, \rn)$ is given by \eqref{eq:sym}, and $\dbar \xi= (2\pi)^{-n} \d \xi$.  In particular, the kernel of the operator $A^\iota_f$ is  given by the oscillatory integral
\begin{equation}
\label{eq:3}
  K_{A_f^\iota} (x,y)=\int e^{i(x-y) \cdot \xi} a^\iota_f(x,\xi) \, \dbar \xi \in \Cinft(W_\iota \times W_\iota). 
\end{equation}
\end{theorem}
\begin{proof} 
Our considerations will essentially follow the proof of Theorem 4 in  \cite{ramacher06}, or Theorem 2 in \cite{parthasarathy-ramacher11}. For  a review on pseudodifferential operators, the reader is referred to \cite{shubin}.  Fix a chart $(\widetilde W_\iota, \phi_{\iota})$, and let   ${p} \in \widetilde W_\iota$, $x =(x_1,\dots,x_n)= \phi_\iota({p})\in \rn$.  In what follows we shall show that  $a_{f}^{\iota}(x,\xi)$ belongs to the symbol class $  \S^{-\infty}(W_\iota \times \rn)$. For later purposes, we  shall actually consider  the slightly more general amplitudes 
\begin{align}
\label{eq:symk}\begin{split}
a_{f}^{\iota\tilde \iota}(x,\xi;k_1,k_2)&=e^{-i\phi_{\tilde \iota}^{k_1k_2}(x)\cdot \xi} 
\alpha'_{\tilde \iota}(k_1 k_2 \cdot \phi_\iota^{-1}(x))  \int_{G}e^{i\phi_{\iota}^{k_1gk_2}(x)\cdot\xi} c_{\iota}(x,k_1gk_2)f(g) \d g \\&=e^{-i\phi_{\tilde \iota}^{k_1k_2}(x)\cdot \xi}  \alpha'_{\tilde \iota}(k_1 k_2 \cdot \phi_\iota^{-1}(x)) \int_{G}e^{i\phi_{\iota}^{g}(x)\cdot\xi} c_{\iota}(x,g)(L(k_1)R(k^{-1}_2)f)(g) \d g,
\end{split}
\end{align}
where   $k_1,k_2 \in G$. Here we took into account the unimodularity of $G$. In particular, $a_{f}^{\iota}(x,\xi)=a_{f}^{\iota \iota}(x,\xi;e,e)$.
Denote by $V_{\iota,{p}}$  the set of all $g\in G$ such that $g \cdot {p} \in \widetilde{W}_\iota$. Assume that  $g \in V_{\iota,{p}}$, and  write
\bqn 
\psi^\iota_{\xi,x} (g)=e^{i\phi_{\iota}^{g}(x)\cdot\xi}.
\eqn
For $X\in\g$ one computes that
\begin{align*}
dL(X)\psi^\iota_{\xi,x}(g)&=\frac{d}{ds} {e^{i\phi_{\iota}^{\e{-sX}g}(x)\cdot\xi} }_{|s=0}=i\psi^\iota_{\xi,x}(g)\sum_{{i}=1}^{n}\xi_{i}dL(X)x_{i,{p}}(g),
\end{align*}
where we put $x_{i,{p}}(g)= x_i(g \cdot {p})$. Let $\mklm{X_1, \dots, X_d}$ be a basis of $\g$. Since $G$ acts locally transitively on $\widetilde W_\iota$, the $n \times d$ matrix 
\bqn
\left ( dL(X_j)x_{i,{p}}(g) \right )_{i,j}
\eqn
has maximal rank. As a consequence, there exists a neighborhood $\widetilde U_p$ of ${p}$, and indices $j_1, \dots, j_n$ such that 
\bqn 
\det \left ( dL(X_{j_k})x_{i,{p}'}(g) \right )_{i,k} \not = 0 \qquad \forall {p} ' \in \widetilde U_p.
\eqn
Hence,
\bq
\label{eq:23}
\begin{pmatrix} dL(X_{j_1})\psi^\iota_{\xi,x'}(g)\\ \vdots\\ dL(X_{j_n})\psi^\iota_{\xi,x'}(g)\end{pmatrix} =i\psi^\iota_{\xi,x'}(g)\mathcal{M}(x',g)\xi,
\eq
where $\mathcal{M}(x',g)=\left ( dL(X_{j_k})x_{i,\phi_\iota^{-1}(x')}(g) \right )_{i,k}\in \GL(n,\R)$ is an invertible matrix for all $x' \in \phi_\iota(\widetilde U_p)$. Consider now the extension of $\mathcal{M}(x',g)$  as an endomorphism in $\C^1[\R^{n}_\xi]$ to the symmetric algebra ${\rm{S}}(\C^1[\R^{n}_\xi])\simeq \C[\R^{n}_\xi]$.  Since  $\mathcal{M}(x',g)$ is invertible, its extension to $ {\rm{S}}^N(\C^1[\R^{n}_\xi])$ is also an automorphism for any $N\in\N$. Regarding  the polynomials $\xi_1,\dots,\xi_{n}$ as a basis in $\C^1[\R^{n}_\xi]$, let us denote the image of the basis vector $\xi_j$ under the endomorphism $\mathcal{M}(x',g)$ by $\mathcal{M} \xi_j$, so that  by \eqref{eq:23}
\begin{align*}
\mathcal{M} \xi_k&= -i  \psi^\iota_{-\xi,x'}(g) dL(X_{j_k})\psi^\iota_{\xi,x'}(g),  \qquad  1\leq k \leq n.
\end{align*}
In this way, each polynomial  $\xi_{j_1} \otimes \dots \otimes \xi_{j_N}\equiv \xi_{j_1} \dots \xi_{j_N}$ can  be written as a linear combination
 \begin{equation}
\label{24}
   \xi^\alpha =\sum _\beta \Lambda^\alpha_\beta (x',g) \M \xi_{\beta_1} \cdots \M \xi_{\beta_{|\alpha|}},
 \end{equation}
where the $\Lambda^\alpha_\beta(x',g)$ are smooth functions given in terms of the matrix coefficients of $\mathcal{M}(x',g)$.  We now have for arbitrary indices $\beta_1,\dots, \beta_r$ and all $x' \in \phi_\iota(\widetilde U_p)$
  \begin{align}
\label{25}
\begin{split}
    i^r \psi^\iota_{\xi,x'}(g) \mathcal{M}\xi_{\beta_1} \cdots \mathcal{M}\xi_{\beta_r}&= dL(X_{\beta_1} \cdots X_{\beta_r}) \psi^\iota_{\xi,x'}(g)\\&+ \sum_{s=1}^{r-1} \sum _{\alpha_1,\dots, \alpha_s} d ^{\beta_1,\dots, \beta_r}_{\alpha_1,\dots, \alpha_s} (x' ,g) dL(X_{\alpha_1} \cdots X_{\alpha_s}) \psi^\iota_{\xi,x'}(g),
\end{split}  
\end{align}
where the coefficients $ d ^{\beta_1,\dots, \beta_r}_{\alpha_1,\dots, \alpha_s} (x' ,g) $ are smooth functions given by the matrix coefficients of $\mathcal{M}(x',g)$ which are  at most of exponential growth  in $g$, and independent of $\xi$, see Lemma 4 in \cite{parthasarathy-ramacher11}.
The key step in proving the theorem is that, as an immediate consequence of equations \eqref{24} and \eqref{25},  we can express $(1+|\xi|^2)^N$  as a linear combination of derivatives $d L(X^\alpha) \psi^\iota_{\xi,x'}(g)$, obtaining for arbitrary $N \in \N$ and $x' \in \phi_\iota(\widetilde U_p)$ the equality 
\begin{equation}
  \label{26}
 \psi^\iota_{\xi,x'}(g)(1+|\xi|^2)^N=  \sum_{r=0}^{2N} \sum_{|\alpha| =r} b^N_\alpha(x',g) d L(X^\alpha) \psi^\iota_{\xi,x'}(g),
\end{equation}
where the coefficients $b^N_\alpha(x',g) $  are at most of exponential growth in  $g$. Let us now show that $  a_f^{\iota\tilde \iota}(x,\xi;k_1,k_2) \in \Syms(W_\iota \times \R^{n}_\xi)$ for each fixed $k_1,k_2 
\in K$. Note that  $ a_f^{\iota\tilde \iota}(x,\xi;k_1,k_2) \in \Cinft(W_\iota \times \R^{n}_\xi\times K\times K)$. While differentiation with respect to $\xi$ does not alter the growth properties of the functions  $ a_f^{\iota\tilde \iota}(x,\xi;k_1,k_2)$, differentiation with respect to $x$  yields additional powers in $\xi$.  As one computes,  $(\gd^\alpha_\xi \gd^\beta_x   a ^{\iota\tilde \iota} _f)(x,\xi;k_1,k_2)$ is a finite sum of terms of the form
\begin{equation*}
  \xi^{\delta} e^{-i\phi_\iota^{k_1k_2}(x) \cdot \xi} \int_{G}  \psi^\iota_{\xi,x}(g) (L(k_1)R(k_2^{-1})f)(g) d_{\beta'\beta''}^\delta (x,k_1,k_2,g)(\gd ^{\beta'}_x c_\iota)(x,g) \gd ^{\beta''}_x [\alpha'_{\tilde \iota}(k_1 k_2 \cdot \phi_\iota^{-1}(x))] \d g,
\end{equation*}
the functions $d_{\beta'\beta''}^\delta(x,k_1,k_2,g) $ being at most of exponential growth in $g$. Let next $f_1\in\S(G)$, and assume that $f_2 \in \Cinft(G)$, together with all its derivatives, is at most of exponential growth. Then, by \cite{ramacher06}, Proposition 1, we have
\begin{equation}
\label{eq:9}
\int_{G}f_1(g) dL(X^\iota) f_2(g) d_{G}(g)=(-1)^{|\iota|} \int _{G}
dL(X^{\tilde \iota}) f_1(g) f_2(g) d_{G}(g),
\end{equation}
where for  $X^\iota=X^{\iota_1}_{i_1}\dots X^{\iota_r}_{i_r}$ we wrote $X^{\tilde \iota}=X^{\iota_r}_{i_r}\dots X^{\iota_1}_{i_1}$, $\iota$ being an  arbitrary multi-index.  
Let now  $\mathcal{O}$ denote an arbitrary  compact set  in $W_\iota$. By Heine--Borel, $\phi_\iota^{-1}(\mathcal{O})$ can be covered by a finite number of neighborhoods $\widetilde U_p$.  
Making use of equation \eqref{26}, and integrating according  to \eqref{eq:9}, we obtain for arbitrary multi-indices $\alpha, \beta$ the estimate
\begin{equation*}
  |(\gd ^\alpha_\xi \gd ^\beta _x  a_f^{\iota\tilde \iota}) (x,\xi;k_1,k_2) | \leq \frac 1 {(1+\xi^2)^N} C_{\alpha,\beta,\mathcal{O}} \qquad x \in \mathcal{O},
\end{equation*}
where $N\in \N$, since $L(k_1) R(k^{-1}_2) f \in \S(G)$. This proves that $ a_f^{\iota\tilde \iota}(x,\xi;k_1,k_2) \in \Syms(W_\iota \times \R^{n}_\xi)$ for each fixed $k_1,k_2 \in K$. Since equation \eqref{eq:2} is an immediate consequence of the Fourier inversion formula, the proof of the theorem is now complete.
\end{proof}

Let $dM$ be a fixed $G$-invariant density on $M$, and denote by $\L^2(M)$ the space of square integrable functions on $M$. 
In case that $M$ is compact, the fact that the integral operators $\pi(f)$ have smooth kernels implies that they are trace-class operators in $\L^2(M)$. Indeed, one has the following
\begin{lemma}
\label{Mia}
Let $\bf X$  be a compact manifold of dimension $n$ with volume form $d\, \bf X$. Let $k:\bf X\times \bf X \to \C$ be a kernel function of class $\mathrm{C}^{(n+1)}(\bf X \times \bf X)$. Then the operator
\[
(Kf)(p) = \int_{\bf X} k(p,q)f(q)d\, {\bf{X}} (q),  \qquad f \in L^2({\bf X}, d \, {\bf X}),
\]
is trace class, and $tr\,K = \int_{\bf X} k(p, p)  \d\, {\bf X} (p)$.
\end{lemma}
\begin{proof}
See \cite{miatello80}, Lemma 2.2.
\end{proof}

In our situation, we obtain 

\begin{corollary}
\label{cor:1}
Let $M$ be a compact, $\Cinft$-manifold of dimension $n$,  and $G$ a connected, real, semisimple Lie group with finite center acting on $M$ in a transitive way. If $f\in \S({G})$, then  $\pi(f)$ is a trace class operator in $\L^2(M)$, and 
\bq
\label{eq:10}
\tr  \pi(f)=\sum_\iota \int_{W_\iota} (\alpha_\iota \circ \phi_{\iota}^{-1})(x) K_{A_f^\iota} (x,x) \d x=\sum_\iota \int_{M} \alpha_\iota (p) K_{A_f^\iota} (\phi_\iota(p),\phi_\iota({p})) j_\iota (p) dM(p),
\eq
where $d x$ denotes Lebesgue measure in $\rn $, and  $(\phi_\iota)^\ast (dx) =j_\iota d M$.

\end{corollary}
\begin{proof}
By Theorem  \ref{thm:2} ,  $\mathcal{K}_f\in \Cinft(M\times M, {{\bf 1}} \boxtimes \Omega_{M})$. Locally, the kernel $\mathcal{K}_f$ is determined by the smooth functions \eqref{eq:3}. Restricting the latter to the respective diagonals in $W_\iota$, one obtains a family of functions on $M$
\bqn 
k_f^\iota ({p})=  K_{A_f^\iota} (\phi_\iota(p),\phi_\iota({p})), \qquad {p} \in\widetilde W_\iota,
\eqn
which define a density  $k_f dM \in \Cinft(M, \Omega_M)$ on $M$. Since $M$ is compact, it can be integrated, and by   Lemma \ref{Mia}  we get 
\begin{align*} 
\tr  \pi(f) &= \int_M k_f (p) \, dM(p)=\sum_\iota \int_{W_\iota} (\alpha_\iota \circ \phi_{\iota}^{-1})(x) K_{A_f^\iota} (x,x) \d x=\sum_\iota \int_{\widetilde W_\iota} \alpha_\iota (p) k_f^\iota (p) j_\iota (p) dM(p),
\end{align*}
where we wrote $(\phi_\iota)^\ast (dx) =j_\iota d M$. 
\end{proof}

\section{Equivariant heat asymptotics}

From now on, let $M$ be a closed, real-analytic Riemannian manifold of dimension $n$,  and $G$ a connected, real, semisimple Lie group with finite center acting transitively and isometrically on $M$. Assume that $M$ is endowed with a $G$-invariant density $dM$. Consider further a maximal compact subgroup $K$ of $G$, and let  $\widehat{K}$ denote the set of all equivalence classes of unitary irreducible representations of $K$.
Let $(\pi_\sigma,V_\sigma)$ be a unitary irreducible representation of $K$  of dimension $d_\sigma$ belonging to $\sigma \in \widehat{K}$, and $\chi_\sigma(k)=\tr \pi_\sigma(k)$ the corresponding character. As a unitary representation of $K$, $(\pi,\L^2(M))$ decomposes into isotypic components according to 
\bqn 
\L^2(M)\simeq \bigoplus_{\sigma \in \widehat{K}} \L^2(M)_\sigma,
\eqn
where  $ \L^2(M)_\sigma=P_\sigma ( \L^2(M))$, and 
$ 
P_\sigma=d_\sigma \int _K \overline{\chi_\sigma(k)} \pi(k) \d k$ 
is the corresponding projector  in $\L^2(M)$, $dk$ being a Haar measure on $K$.  Let $f \in \S(G)$, and consider the restriction $P_\sigma \circ \pi(f) \circ P_\sigma$ of the integral operator $\pi(f)$ to the isotypic component $\L^2(M)_\sigma$. As one computes, for  $\phi \in \L^2(M)$,
\begin{align*}
[P_\sigma \circ \pi(f) \circ P_\sigma] \phi(p)&=d^2_\sigma \int_K \overline{\chi_\sigma(k)} [\pi(f) \circ P_\sigma ] \phi(k \cdot p)\d k\\
&= d^2_\sigma \int_K \int_G \overline{\chi_\sigma(k)} f(g) P_\sigma \phi (g k \cdot p) \d g \d k\\
&= d^2_\sigma \int_K \int_G \int_K \overline{\chi_\sigma(k)} f(g)\overline{\chi_\sigma(k_1)} \phi (k_1  g  k \cdot p) \d k_1 \d g \d k.
\end{align*}
 Since $G$ is unimodular, one obtains
\bq
\label{eq:11}
P_\sigma \circ \pi(f) \circ P_\sigma= \pi(H^\sigma_f),
\eq
where $H^\sigma_f \in \S(G)$  is given by
\bq
\label{eq:11a}
H^\sigma_f(g)=d^2_\sigma \int_K \int_K f(k_1^{-1} g k^{-1}) \overline{\chi_\sigma(k_1)} \overline{\chi_\sigma(k)}d k \d k_1. 
\eq
Clearly,  $H^\sigma_f \in \S(G)$, compare \cite{barbasch-moscovici83}, Proposition 2.4.
Note that if $f$ is $K$-bi-invariant, $\pi(f) $ commutes with $P_\sigma$, so that $P_\sigma \circ \pi(f) \circ P_\sigma=P_\sigma \circ \pi(f)=\pi(f) \circ P_\sigma$.
In Section 5, we shall also consider kernels of the form
\bqn 
\int_K \int_K f(k^{-1}_1 g k^{-1}) \sigma_{ij}(k)\sigma_{lm}(k_1)\d k \d k_1
\eqn
where  $\sigma_{ij}(k)=\eklm{e_i,\pi_\sigma(k) e_j}$ are matrix elements of $\sigma$ with respect to a basis $\mklm{e_i}$ of $V_\sigma$.  With the notation  as  in the previous section we now have   the following

\begin{proposition}
\label{prop:1}
Let  $f \in \S(G)$, and $\sigma \in \widehat{K}$.  Then   $\pi(H^\sigma_f)$ is of trace class, and 
\begin{align*}
\tr \,  \pi(H^\sigma_f)=&\frac{d_\sigma^2}{(2\pi)^n} \sum_{\iota, \tilde \iota}    \int_K  \int_K  \int_{T^\ast M}    e^{i\Phi_{\iota\tilde \iota} (p,\xi, k_1,k) } \alpha_\iota ( p)   \alpha_{\tilde \iota}( k_1k \cdot p )  \overline{\chi_\sigma(k_1)} \overline{\chi_\sigma(k)}   \\
&\cdot  a_{f}^{\iota\tilde \iota} (\phi_\iota( p),\xi;k_1,k)  j_\iota(p) d(T^\ast M)(p,\xi)\d k \d k_1,
\end{align*}
where $d(T^\ast M)(p,\xi)$ denotes the canonical density on the cotangent bundle $T^\ast M$, and we set  
$$
\Phi_{\iota\tilde \iota} (p,\xi,k_1, k)= (\phi_{\tilde \iota}(k_1k \cdot p)- \phi_\iota(  p)) \cdot \xi,
$$
 while  $ a_f^{\iota\tilde \iota}(x,\xi;k_1,k_2) \in \Syms(W_\iota \times \R^{n}_\xi)$ was defined in \eqref{eq:symk}. 
\end{proposition}
\begin{proof}
By Corollary \ref{cor:1},  $\pi(H^\sigma_f)$ is of trace class, and at the microlocal level  one has 
\begin{align*}
\Big [ \pi(H^\sigma_f) ( u \circ \phi_\iota) \Big ] ( \phi_\iota^{-1} (x))=A_{H^\sigma_f}^\iota u(x), \qquad u \in \CT(W_\iota), 
\end{align*}
where $A_{H^\sigma_f}^\iota$ is given by \eqref{eq:2}. By the unimodularity of $G$, together with \eqref{eq:sym} and \eqref{eq:3}, 
\begin{align*}
K_{A_{H^\sigma_f}^\iota}(x,y)&=\int e^{i(x-y)\cdot \xi} a_{H^\sigma_f}^\iota(x,\xi)\, \dbar \xi
= \int \left [\int_{G}e^{i(\phi_{\iota}^{g}(x)-y)\cdot\xi} c_{\iota}(x,g)H^\sigma_f(g) \d g \right ] \, \dbar \xi \\
&= d^2_\sigma\int \left [ \int_G  \int_K \int_K f( g)  e^{i(\phi_{\iota}^{k_1gk}(x)-y)\cdot\xi} c_{\iota}(x,k_1gk)\overline{\chi_\sigma(k_1)} \overline{\chi_\sigma(k)} \d k \d k_1 \d g\right ] \, \dbar \xi.
\end{align*}
Let $\psi \in \CT(\rn,\R^+)$ be equal $1$ near the origin, and $\epsilon >0$. By Lebesgue's theorem on bounded convergence,
\bqn 
K_{A_{H^\sigma_f}^\iota}(x,y)=\lim_{\epsilon \to 0} \int e^{i(x-y)\cdot \xi} a_{H^\sigma_f}^\iota(x,\xi) \psi(\epsilon \xi) \, \dbar \xi,
\eqn
since $a_{H^\sigma_f}^\iota(x,\xi)$ is rapidly falling in $\xi$. 
Arguing as in the proof of Corollary \ref{cor:1}, one obtains for $\tr  \pi(H^\sigma_f)$ the expression
\begin{gather*}
\lim _{\epsilon \to 0} d_\sigma^2 \sum_\iota   \int_{\widetilde W_\iota} \int  \int_G  \int_K \int_K  e^{i(\phi_{\iota}(k_1gk\cdot p)-\phi_\iota(p))\cdot\xi}  f( g) \alpha_\iota (p)    c_{\iota}(\phi_\iota(p),k_1gk)\overline{\chi_\sigma(k_1)} \overline{\chi_\sigma(k)} \psi(\epsilon \xi) \\
\cdot  j_\iota(p)   \d k \d k_1 \d g \, \dbar \xi \d M(p)\\
=\lim _{\epsilon \to 0} d_\sigma^2 \sum_{\iota,\tilde \iota}  \int_{ W_\iota} \int  \int_G  \int_K \int_K   e^{i(\phi_{\iota}(k_1 g k \cdot p)- \phi_{\tilde \iota}(k_1k  \cdot p)) \cdot \xi} e^{i(\phi_{\tilde \iota}(k_1k \cdot p)-\phi_\iota( p))\cdot\xi}f(g)  \alpha_\iota ( p)     \\
\cdot   \alpha'_\iota (k_1 g k \cdot p) \alpha _{\tilde \iota}( k_1k \cdot p ) \alpha'_{\tilde \iota} (k_1k  \cdot p)\overline{\chi_\sigma(k_1)} \overline{\chi_\sigma(k)} \psi(\epsilon \xi)  j_\iota(  p)  \,  \d k \d k_1 \d g \dbar \xi \d M(p) \\
=\lim _{\epsilon \to 0}d_\sigma^2 \sum_{\iota, \tilde \iota}    \int_K  \int_K  \int_{\widetilde W_\iota} \int   e^{i(\phi_{\tilde \iota}(k_1 k\cdot p)-\phi_\iota(p))\cdot\xi} \alpha_\iota (p)  \alpha_{\tilde \iota} (k_1k \cdot p)  \overline{\chi_\sigma(k_1)} \overline{\chi_\sigma(k)}  \psi(\epsilon \xi) \\
\cdot  a_{f}^{\iota\tilde \iota} (\phi_\iota(p),\xi;k_1,k)  j_\iota(p)  \, \dbar \xi \d M(p)\d k \d k_1, 
\end{gather*}
where  the change of  order of integration is permissible, since everything is absolutely convergent.  
Note that we used the equality
\bqn 
1= \sum_{\tilde \iota} \alpha _{\tilde \iota}( k_1k \cdot p ) \alpha'_{\tilde \iota} (k_1k \cdot p). 
\eqn 
 Finally, it was shown in the proof of Theorem \ref{thm:2} that  $a_{f}^{\iota\tilde \iota} (\phi_\iota(p),\xi;k_1,k)$ is rapidly falling in $\xi$,  so that we can pass to the limit under the integral, and the assertion  follows.
\end{proof}

In what follows, we shall address  the case where $f=f_t\in \S(G)$, $t >0$, is the Langlands kernel of a semigroup generated by a strongly elliptic operator associated to the representation $\pi$. Our main goal will  be the derivation of  asymptotics for 
\bqn
\tr   \pi(H^\sigma_{f_t})= \tr  (P_\sigma \circ \pi(f_t)\circ P_\sigma) 
\eqn
 as $t \to 0^+$. Thus, let $\G$ be a Lie group and $(\pi, \B)$ a continuous representation  of $\G$ in some Banach space $\B$. Denote by  $\g$ the Lie algebra of $\G$, and by $X_1, \dots, X_d$ a basis of it. Consider further a strongly  elliptic differential operator of order $q$ associated to $\pi$
\begin{equation}\label{strellop}
\Omega = \sum_{|\alpha| \le q} c_\alpha d\pi(X^\alpha),
\end{equation}
meaning that  $\Re (-1)^{q/2}\displaystyle\sum_{\alpha = q} c_\alpha \xi^\alpha \ge \kappa|\xi|^q$ for all $\xi \in \R^d$, and some $\kappa >0$. The general theory of strongly continuous semigroups establishes that its closure generates a strongly continuous holomorphic semigroup of bounded operators which is given by
\begin{equation}\label{smgrpstbndop}
S_\tau=\frac{1}{2\pi i} \int_{\Lambda} e^{\lambda \tau}(\lambda  \mathbbm{1}+ \overline{\Omega})^{-1}d\lambda,
\end{equation}
where $\Lambda$ is an appropiate path in $\C$ coming from infinity and going to infinity, and $|\arg \tau| < \eta$ for an appropiate $\eta \in (0, \pi/2]$. The integral converges uniformly with respect to the operator norm, and for $t>0$,  the semigroup $S_t$ can be characterized by a convolution semigroup $\mklm{\mu_t}_{t>0}$ of complexes measures on $\G$ according to
\[
S_t = \int_\G\pi(g)d\mu_t (g),
\]
the representation $\pi$ being measurable with respect to the measures $\mu_{t}$. The $\mu_t$ are absolutely continuous  with respect to Haar measure $d_\G$ on $\G$ so that, if we denote by $f_t(g) \in L^1(\G, d_\G)$ the corresponding Radon-Nikodym derivatives,  one has an expressions
\begin{equation}\label{smgrlangkern}
S_t = \pi(f_t)= \int_G f_t(g)\pi(g)d_\G(g), \qquad t>0. 
\end{equation}
The  function $f_t(g) \in L^1(\G, d_\G)$ is analytic in $t \in \R^+_\ast $ and $g \in \G$,  universal for all Banach representations, and one can show that $f_t \in \S(\G)$. Moreover, it satisfies  the following $\L^\infty$ upper bounds. There exist constants $a,b,c_1,c_2>0$ and $\omega \geq 0$ such that 
\begin{equation}
\label{eq:boundlang}
|(dL(X^\alpha) \partial^l_t f_t)(g)|_{t=\tau} \le ac_1^{|\alpha|} c_2^l |\alpha|! \, l!\,  \tau^{-\frac{|\alpha|+d}q -l}  e^{\omega \tau} 
e^{-b(|g|^q/\tau)^{1/(q-1)}}
\end{equation}
for all $\tau>0$, $g \in G$, $ l \in \N$, and multi-indices $\alpha$. For a complete exposition of these facts the reader is referred  to \cite{robinson}, pages 30, 152, and 209, or \cite{ter_elst-robinson}. In what follows, we shall call $f_t$ the \emph{Langlands, or group kernel} of the holomorphic semigroup $S_t$.  Returning to our situation, let  $\G=G$, and $\pi$ be  the regular representation of $G$ on $\L^2(M)$. Let us mention that as a
consequence of the bounds  \eqref{eq:boundlang}, we have the following 
\begin{corollary}
There exist constants $a,b,c_1,c_2>0$ and $\omega \geq 0$ such that 
\begin{equation*}
|(dL(X^\alpha) \partial^l_t H^\sigma _{f_t})(g)|_{t=\tau} \le ac_1^{|\alpha|} c_2^l |\alpha|! \, l!\,  \tau^{-\frac{|\alpha|+d}q -l}  e^{\omega \tau} 
e^{-b(\frac{d(gK,K)^q}{\tau})^{1/(q-1)}}
\end{equation*}
for all $\tau>0$, $g \in G$, $ l \in \N$, and multi-indices $\alpha$. 
\end{corollary}
\begin{proof}
 Clearly, 
\[
| H^\sigma _{f_t}(g)|_{t=\tau}\le d^2_\sigma \int_K \int_K | f_t (k_1^{-1} g k^{-1}))|  \d k \d k_1. 
\]
According to  \eqref{eq:boundlang} we therefore have
\[
|H^\sigma_{f_t}(g)| \le d^2_\sigma a  \,   t^{-\frac{d}q }  e^{\omega t} \int_K \int_K \exp\left(-b\left(\frac{|k_1^{-1}g k^{-1}|^q}{t}\right)^{\frac{1}{q-1}}\right) \d k \d k_1,
\]
where $|k_1^{-1}g k^{-1}| = d(k_1^{-1}g k^{-1}, e)= d(g k^{-1}, k_1)$. Put $\X=G/K$, and let $\g=\k\oplus \p$ be a Cartan decomposition of $\g$.  By restriction of the Killing form to $T_e\X \simeq \p$ one obtains an invariant Riemannian metric on $\X$ such that the canonical projection map $G \rightarrow \X$ becomes a Riemannian submersion. Now, if $d(gK,hK)$ denotes the geodesic distance on $\X$, 
\[
|g|= d(g,e) \ge d(gK, K), \qquad g \in G,
\]
compare  \cite{mueller98}, Theorem 3.1.
By applying similar arguments to the derivatives, the corollary follows.
\end{proof}

Let $\beta \in \CT(G)$, $0 \leq \beta \leq 1$ have support in a sufficiently small neighborhood $U$ of $e \in G$ satisfying $U=U^{-1}$, and assume that  $\beta=1$ close to $e$. We then have the following

\begin{theorem}
\label{thm:A}
Consider  a strongly elliptic differential operator $\Omega$ of order $q\geq 2$ associated to $(\pi, \L^2(M))$, and  the corresponding semigroup $S_t=\pi(f_t)$ with Langlands kernel $f_t$, $t>0$. Let $\sigma \in \widehat{K}$.   Then 
$$\tr   \pi(H^\sigma_{f_t})=\tr  \pi(H_{f_t\beta}^\sigma )+O(t^\infty),$$
where 
 \begin{gather*}
\tr  \pi(H_{f_t\beta}^\sigma )=\frac {d_\sigma^2} {(2\pi)^nt^{n/q}} \sum_{\iota} \int_K  \int_K \int _{T^\ast M} e^{i\Phi_{\iota\iota}(p,\xi,k_1,k)/t^{1/q}}  \alpha_\iota ( p)  \overline{\chi_\sigma(k_1)\chi_\sigma(k) }   \\ \cdot  \, b_{f_t}^{ \iota}(\phi_\iota(p),\xi/t^{1/q};k_1,k) j_\iota(p)  \d(T^\ast M)(p,\xi)\d k \d k_1, \qquad t>0,
\end{gather*}
and
\bqn 
b_{f_t}^{\iota}(\phi_\iota(p),\xi;k_1,k)=e^{-i\phi_{ \iota}({k_1k\cdot p )}\cdot \xi}   \int_{U}e^{i\phi_{\iota}(k_1gk \cdot p)\cdot\xi} c_{\iota}(\phi_\iota(p),k_1gk)f_t(g)\beta(g) \d g
\eqn
is rapidly decaying in $\xi$, and vanishes if $k_1 k \cdot p \not \in \tilde W_\iota'$. Furthermore, for any multi-indices $\alpha,\beta, \delta_1,\delta$
$$
|\gd_x^\alpha \gd ^\beta_\xi \gd _{k_1}^{\delta_1} \gd _k^\delta [b_{ f_t}^\iota(x,\xi/t^{1/q};k_1,k)]|  \leq C
$$
 for some constant $C>0$ independent of $0<t<1$. \end{theorem}
\begin{proof}
 To determine the asymptotic behavior of $\tr   \pi(H^\sigma_{f_t})$ as $t \to 0$ by means of Proposition \ref{prop:1}, we first have to examine the $t$-dependence of the amplitude
$a_{ f_t}^{\iota\tilde \iota}(\phi_\iota(p),\xi;k_1,k)$ as $t \to 0$ for fixed $k,k_1 \in K$.
Let  $0 \leq \beta \leq1 $ be a test function on $G$ with support in a sufficiently small neighborhood $U=U^{-1}$ of the identity  that is identically 1 on a ball of radius $R>0$ around $e$, and consider for  $f \in \S(G)$ 
\begin{align*}
\, ^1 a_{ f}^{\iota\tilde \iota}(x,\xi;k_1,k_2)&=e^{-i\phi_{\tilde \iota}^{k_1k_2}(x)\cdot \xi} 
\alpha'_{\tilde \iota}(k_1 k_2 \cdot \phi_\iota^{-1}(x))  \int_{G}e^{i\phi_{\iota}^{k_1gk_2}(x)\cdot\xi} c_{\iota}(x,k_1gk_2)f(g)  \\ &\cdot (1- \beta)(g) \d g, \\
\, ^2 a_{ f}^{\iota\tilde \iota}(x,\xi;k_1,k_2)&=e^{-i\phi_{\tilde \iota}^{k_1k_2}(x)\cdot \xi} 
\alpha'_{\tilde \iota}(k_1 k_2 \cdot \phi_\iota^{-1}(x))  \int_{G}e^{i\phi_{\iota}^{k_1gk_2}(x)\cdot\xi} c_{\iota}(x,k_1gk_2)f(g)  \\ & \cdot  \beta(g) \d g.
\end{align*}
Similarly to \eqref{26}, one has  for arbitrary $N \in \N$  the equality 
\begin{equation}
\label{eq:2.2.2012}
 \psi^\iota_{\xi,x}(k_1gk_2)(1+|\xi|^2)^N=  \sum_{r=0}^{2N} \sum_{|\alpha| =r} b^N_\alpha(x,g,k_1,k_2) d L(X^\alpha) [\psi^\iota_{\xi,x}(k_1gk_2)],
\end{equation}
where the coefficients $b^N_\alpha(x,g,k_1,k_2) $  are at most of exponential growth in  $g$. With  \eqref{eq:9} and \eqref{eq:boundlang}  we obtain
\begin{align*}
|\, ^1 a_{ f_t}^{\iota\tilde \iota}(\phi_\iota(p),\xi;k_1,k)|\leq c (1+|\xi|^2)^{-N} e^{\omega t} t^{-(d+2N)/q} e^{-bR^{q/(q-1)} [t^{1-q} -1]} \int_G  e^{- b |g|^{q/(q-1)}} e^{\kappa |g|}\d g
\end{align*}
 for small $t>0$, and constants $b,c>0$, $\kappa , \omega\geq 0$. Consequently, $\, ^1 a_{ f_t}^{\iota\tilde \iota}(x,\xi;k_1,k)$ vanishes to all orders  as $t \to 0$, or $|\xi| \to \infty$, provided that $q\geq 2$, and with Proposition \ref{prop:1} we obtain the equality
\begin{gather*}
\tr  \pi(H_{f_t}^\sigma )= \tr  \pi(H_{ f_t\beta }^\sigma )+O(t^N)\\
= \frac{d_\sigma^2}{(2\pi)^n} \sum_{\iota, \tilde \iota} \int_K  \int_K \int _{T^\ast M} e^{i\Phi_{\iota\tilde \iota}(p,\xi,k_1,k)}  \alpha_\iota ( p)\alpha_{\tilde \iota} (k_1 k \cdot  p)  \overline{\chi_\sigma(k_1)\chi_\sigma(k)}   \\ \cdot  \, ^2a_{f_t}^{\iota\tilde \iota}(\phi_\iota(p),\xi;k_1,k) j_\iota(p)  \d(T^\ast M)(p,\xi)\d k \d k_1 + O(t^N)
\end{gather*}
for any $N \in \N$. Let  $\psi \in \CT(\rn,\R^+)$ be equal $1$ near the origin, and $\epsilon >0$. Repeating the arguments in the proof of Proposition \ref{prop:1} with $f$ replaced by $f_t \cdot \beta$ one obtains  for $\tr  \pi(H_{f_t\beta }^\sigma )$ the expression
\begin{gather*}
\lim _{\epsilon \to 0} d_\sigma^2 \sum_\iota   \int_{\widetilde W_\iota} \int  \int_U  \int_K \int_K  e^{i(\phi_{\iota}(k_1gk\cdot p)-\phi_\iota(p))\cdot\xi}  f_t( g) \beta(g) \alpha_\iota (p)    c_{\iota}(\phi_\iota(p),k_1gk)\overline{\chi_\sigma(k_1)\chi_\sigma(k)} \psi(\epsilon \xi) \\
\cdot  j_\iota(p)   \d k \d k_1 \d g \, \dbar \xi \d M(p)\\
=\lim _{\epsilon \to 0} d_\sigma^2 \sum_{\iota}  \int_U  \int_K \int_K  \int_{ \widetilde W_\iota} \int  e^{i(\phi_{\iota}(k_1 gk  \cdot p)- \phi_{ \iota}(k_1k \cdot p)) \cdot \xi} e^{i(\phi_{\iota}(k_1k \cdot p)-\phi_\iota( p))\cdot\xi}f_t(g)\beta(g)  \alpha_\iota ( p)     \\
\cdot   \alpha'_\iota (k_1 gk \cdot p)  \overline{\chi_\sigma(k_1)\chi_\sigma(k)}  \psi(\epsilon \xi)  j_\iota(p)  \, \dbar \xi \d M(p)\d k \d k_1 \d g.
\end{gather*}
Here we took into account that $U\subset G$ can be chosen so small that for all $k_1,k_2 \in K$, $g \in U$, and $\iota \in I$
\bqn 
k_1 g k_2 \cdot p \in \supp \alpha_\iota' \quad \Longrightarrow \quad k_1 k_2\cdot p \in \widetilde W_\iota',
\eqn 
since $I$ can be assumed to be finite due to the compactness of $M$. Consequently
\bqn 
b_{f}^{\iota}(x,\xi;k_1,k_2)= e^{-i\phi_{ \iota}^{k_1k_2}(x)\cdot \xi}   \int_{U}e^{i\phi_{\iota}^{k_1gk_2}(x)\cdot\xi} c_{\iota}(x,k_1gk_2)f(g)\beta(g) \d g, \qquad f \in \S(G),
\eqn
is well defined, and  $ b_{f_t}^{ \iota}(\phi_\iota(p),\xi;k_1,k)=0$ for $k_1 k \cdot p \not \in \widetilde W_\iota'$.
From the considerations in the proof of Theorem \ref{thm:2}, and \eqref{eq:2.2.2012} it follows that $b_{f}^{\iota}(x,\xi;k_1,k_2) \in S^{-\infty}(W_\iota\times \rn)$ for arbitrary $k_1,k_2 \in K$. 
Thus, we arrive at
\begin{align*}
\tr  \pi(H_{f_t\beta }^\sigma ) &=\lim _{\epsilon \to 0}d_\sigma^2 \sum_{\iota}    \int_K  \int_K  \int_{ \widetilde W_\iota} \int   e^{i(\phi_{ \iota}(k_1k \cdot p)-\phi_\iota( p))\cdot\xi} \alpha_\iota (p)    \overline{\chi_\sigma(k_1)\chi_\sigma(k)}  \psi(\epsilon \xi) \\
&\cdot  b_{f_t}^{\iota} (\phi_\iota( p),\xi;k_1,k)  j_\iota(p)  \, \dbar \xi \d M(p)\d k \d k_1.
\end{align*}
By passing to the limit under the integral, and performing  the substitution $\xi \to \xi/t^{1/q}$, one finally arrives at the desired result. 
To examine the $t$-dependence of the amplitude $b_{f_t}^{ \iota}(\phi_\iota(p),\xi/t^{1/q};k_1,k)$ as $t \to 0$,  introduce canonical coordinates on $U$ according to 
 \bq
 \label{eq:cc}
\Psi: \R^d \ni \zeta =( \zeta_1,\dots, \zeta_d) \longmapsto g=e^{\sum \zeta_i X_i} \in U.
 \eq
 By the analyticity of the $G$-action on $M$ we have the power expansion
\bqn
[\phi_{ \iota}^{k_1k}(x) -\phi_\iota^{k_1gk}(x) ] _j = \sum_{|\alpha|>0} c_\alpha^j(x,k_1,k) \zeta^\alpha, \qquad g \in U,x \in W_\iota,
\eqn
where the coefficients $c_\alpha^j(x,k_1,k)$ depend analytically on $x$, $k_1$,  and $k$. 
Performing the substitution $\zeta \mapsto t^{1/q} \zeta$, and taking into account the bounds \eqref{eq:boundlang}, one computes
\begin{align*}
|b_{f_t}^{ \iota}(x,\xi/t^{1/q};k_1,k)|&=t^{d/q}  \Big | \int_{t^{-1/q} \Psi^{-1}(U)}  e^{i \sum_{|\alpha|>0, j} c^j_\alpha(x,k_1,k) (t^{1/q} \zeta)^\alpha \xi_j/ t^{1/q}}  c_\iota(x, k_1 e^{t^{1/q} \sum \zeta_i X_i} k) \\ 
& \cdot  ( f_t \beta)(e^{t^{1/q} \sum \zeta_i X_i}) \Psi^\ast(d_G)(\zeta) \Big | \leq c' e^{\omega t} \int_{\R^d} e^{-b' |\zeta|^{q/(q-1)}}  \Psi^\ast(d_G)(\zeta)
\end{align*}
for some constants $b',c'>0$, and $\omega\geq 0$, where we took into account that there exists some constant $C>0$ such that $C^{-1} |\zeta| \leq |g| \leq C |\zeta|$.  A similar  examination of the derivatives finally yields for small $t>0$ the estimate
\bqn
|\gd_x^\alpha \gd ^\beta_\xi \gd _{k_1}^{\delta_1} \gd _k^\delta  [b_{ f_t}^\iota(x,\xi/t^{1/q};k_1,k)]|  \leq C
\eqn
for some constant $C>0$ independent of $0<t<1$, and arbitrary indices $\alpha, \beta, \delta_1, \delta$. 
\end{proof}
\begin{remark}
Note that since $b_{f_t}^{ \iota}$ is rapidly decaying in $\xi$, for any $N \in \N$ there exists a constant $c_N>0$ such that 
\bqn 
|b_{f_t}^{ \iota}(\phi_\iota(p),\xi/t^{1/q};k_1,k)|\leq \frac {c_N}{(1 + |\xi/t^{1/q}|^2)^N}=\frac{c_N t^{2N/q}}{( t^{2/q} + |\xi|^2)^N}\leq \frac{c_N t^{2N/q}}{ |\xi|^{2N}}.
\eqn
Therefore, if  $\theta \in \CT(\rn,[0,1])$ is a cut-off function such that $\theta(\xi)=1$ for $|\xi| \leq 1$, and $\theta(\xi)=0$ for $|\xi| \geq 2$, then
\bqn
 \int _K \int _K \int_{T^\ast M} |b_{f_t}^{ \iota}(\phi_\iota(p),\xi/t^{1/q};k_1,k)| (1-\theta(\xi)) d(T^\ast M)(p,\xi) \d k \d k_1 \leq C_N t^{2N/q}
\eqn
for any $N \in \N$, and suitable constants $C_N$.
\end{remark}
 Let us now regard the compact group 
$$
\K=K\times K
$$
with Haar measure $d_{\K}=d_K d_K$. Take  $\sigma \in \widehat{K}$, and $(\pi_\sigma,V_\sigma)\in \sigma$. Then $(\pi_\sigma\otimes \pi_\sigma, V_\sigma \otimes V_\sigma)$ is an unitary irreducible representation of $\K$ belonging to $\sigma \otimes \sigma \in \hat \K$ of dimension $d_{\sigma \otimes \sigma}=d_\sigma^2$, and the corresponding character is given  by 
\bqn 
(\chi_\sigma\otimes \chi_\sigma)(k_1,k)=\chi_\sigma(k_1) \otimes \chi_\sigma(k)=\chi_\sigma(k_1) \chi_\sigma(k), \qquad k_1,k \in K.
\eqn
In what follows, we shall also  write $(k_1,k) \cdot p=k_1 k\cdot p$ for the $\K$-action on $M$. Note that this action is still isometric, but no longer effective. 
\begin{corollary}
\label{cor:3}
Let  $\sigma \in \widehat{K}$,  and $\psi \in \CT(\rn,\R^+)$ be equal $1$ near the origin. Let further  $t, \epsilon >0$.   Then
 \begin{gather*}
\tr   \pi(H^\sigma_{f_t})=\lim_{\eps \to 0} \frac {d_{\sigma\otimes \sigma}} {(2\pi)^nt^{n/q}} \sum_{\iota} \int_{\K} \int _{T^\ast M} e^{i\Phi_{\iota\iota}(p,\xi,k_1,k)/t^{1/q}}  \alpha_\iota ( p)  \overline{(\chi_\sigma\otimes \chi_\sigma)(k_1,k)} \\ \cdot  \, b_{f_t}^{ \iota}(\phi_\iota(p),\xi/t^{1/q};k_1,k) \psi(\eps \xi) j_\iota(p)  \d(T^\ast M)(p,\xi) d_{\K}(k_1,k) +O(t^{\infty}).
\end{gather*}
\end{corollary}
\begin{proof}
This is an immediate consequence of Theorem \ref{thm:A}, and Lebesgue's theorem on bounded convergence.
\end{proof}

\section{Singular equivariant asymptotics and resolution of singularities}
\label{sec:SEARS}

The considerations of the previous section showed that, in order to describe the traces $\tr   \pi(H^\sigma_{f_t})= \tr  (P_\sigma \circ \pi(f_t) \circ P_\sigma)$ as $t \to 0^+$, one has to study the asymptotic behavior of oscillatory integrals of the form 
\begin{align}
\label{eq:15}
I(\mu)
&=  \int_{\K}\int _{T^\ast \widetilde{W}}   e^{i  \Phi(p , \xi,k_1,k)/\mu }   a( k_1k \cdot p,  p , \xi,k_1,k) \d(T^\ast M)( p,  \xi) d_{\K}(k_1,k)   
\end{align}
as $ \mu \to 0^+$  by means of the  stationary phase principle, where $(\phi,\widetilde W)$ are local coordinates on $M$,   while $a \in \CT(\widetilde W \times T^\ast  \widetilde W\times \K)$ is an amplitude which might depend on $\mu$, and
\bq
\label{eq:phase}
\Phi(p, \xi, k_1,k) =(\phi((k_1,k)\cdot p) - \phi ( p)) \cdot  \xi.
\eq  
 Consider for this the cotangent bundle  $\pi:T^\ast M\rightarrow M$, as well as the tangent bundle $\tau: T(T^\ast M)\rightarrow T^\ast M$, and define on $T^\ast M$ the Liouville form 
\bqn 
\Theta(\mathfrak{X})=\tau(\mathfrak{X})[\pi_\ast(\mathfrak{X})], \qquad \mathfrak{X} \in T(T^\ast M).
\eqn
Regard $T^\ast M$ as a  symplectic manifold with symplectic form 
$
\omega= d\Theta
$, 
and define for any  element $X$ in the Lie algebra $\k\oplus\k$ of $\K$ the function
\bqn
J_X: T^\ast M \longrightarrow \R, \quad \eta \mapsto \Theta(\widetilde{X})(\eta),
\eqn
where $\widetilde X$ denotes the fundamental vector field on $T^\ast M$, respectively $M$,  generated by $X$.   $\K$ acts on $T^\ast M$ in a Hamiltonian way, and the corresponding symplectic momentum  map is  given by 
\bqn
\mathbb{J}:T^\ast M\to (\k\oplus \k)^\ast,  \quad \mathbb{J}(\eta)(X)=J_X(\eta).
\eqn
Let us next  compute the critical set of the phase function $\Phi$.  Clearly, $\gd_\xi \Phi(p,\xi,k_1,k)=0$ if, and only if $k_1k\cdot p=p$. Write  $\phi(p)=({x}_1,\dots, {x}_n)$,  $\eta=\sum \xi_i (d{x}_i)_p \in T_p^\ast \widetilde W$. Assuming that $k_1k\cdot p= p$, one  computes  for any $X \in \k\oplus \k$
\begin{align*}
\frac d{ds} \Big ( \phi (\e{-tX}  ( k_1, k)  \cdot p)  \cdot \xi \Big )_{|t=0}= \sum \xi_i \widetilde X_{p}({x}_i)   = \sum \xi_i (d{x}_i)_{ p}(\widetilde X_{ p})=\eta (\widetilde X_{p})=\Theta(\widetilde X)(\eta) = \mathbb{J}(\eta)(X),
\end{align*}
so that  $\gd_{(k_1,k)} \Phi(p,\xi,k_1,k)=0$ if, and only if $\mathbb{J}(\eta)=0$. 
A further computation shows that 
\bqn
\gd_x \Phi ( \phi ^{-1} (x), \xi, k_1,k) = [ \, ^T ( \phi \circ k_1k \circ \phi ^{-1} )_{\ast, x} - \1 ]  \xi= ( (k_1 k) ^\ast_x-\1) \cdot \xi, 
\eqn
so that $\gd _p \Phi ( p, \xi, k_1,k )=0$ amounts precisely to the condition $(k_1k)^\ast  \xi=\xi$. 
Collecting everything together one obtains  
 \begin{align}
 \label{eq:16}
 \begin{split}
 \Crit(\Phi)
&=\mklm{(p,\xi,k_1,k) \in T^\ast \widetilde W \times \K: (\Phi_{\ast})_{(p,\xi,k_1,k)}=0} \\
&= \mklm{( p  ,\xi, k_1,k) \in (\Xi  \cap T^\ast \widetilde W)\times \K:  \,  (k_1,k) \cdot (p,\xi)=(p,\xi)},
\end{split}
\end{align}
  where $
  \Xi=\mathbb{J}^{-1}(0)
$
denotes  the zero level of the momentum  map of $\K$. 
Now, the major difficulty resides in the fact that, unless the $\K$-action on $T^\ast M$ is free, the considered momentum map is not a submersion, so that $\Xi$ and  $\Crit(\Phi)$ are not  smooth manifolds. The stationary phase theorem can therefore not immediately be applied to the integrals  $I(\mu)$. Nevertheless, it was shown in \cite{ramacher10}   that by constructing a strong resolution of the set 
\bqn 
\Ncal=\mklm{(p, k_1,k) \in M \times \K: (k_1,k) \cdot p = p}
\eqn
a partial desingularization $\mathcal{Z}: \widetilde {\bf X} \rightarrow {\bf X}= T^\ast M \times \K$ of  the  set 
\bqn 
\mathcal{C}=\mklm{( p  ,\xi, k_1,k) \in \Xi \times \K:  \,  (k_1,k) \cdot (p,\xi)=(p,\xi)}
\eqn
can be achieved, and applying the stationary phase theorem in the  resolution space, an asymptotic description of $I(\mu)$ can be obtained. More precisely, the map $\mathcal{Z}$ yields 
a partial monomialization of  the local ideal $I_\Phi=(\Phi)$ generated by the phase function \eqref{eq:phase} according to  
\bqn 
\mathcal{Z}^\ast (I_\Phi) \cdot \E_{\tilde x, \widetilde {\bf X}} = \prod_j\sigma_j^{l_j}   \cdot\mathcal{Z}^{-1}_\ast(I_\Phi) \cdot \E_{\tilde x, \widetilde {\bf X}},
\eqn
 where $\E_{\widetilde{\bf X}}$ denotes the structure sheaf of rings of $\widetilde{\bf X}$,  $\sigma_j$ are local coordinate functions near each $\tilde x \in \widetilde {\bf X}$, and $l_j$ natural numbers. As a consequence, the phase function factorizes locally according to $\Phi \circ \mathcal{Z} \equiv \prod \sigma_j^{l_j} \cdot  \tilde \Phi^ {wk}$,
and one shows  that  the weak transforms $ \tilde \Phi^ {wk}$ have clean critical sets. 
Asymptotics for the integrals $I(\mu)$ are  then obtained by pulling  them back to the resolution space $\widetilde {\bf X}$, and applying  the stationary phase theorem to the $\tilde \Phi^{wk}$  with the variables  $\sigma_j$ as parameters. As a consequence, one obtains 
\begin{theorem} 
\label{thm:I(mu)}
Let $M$ be a connected, closed Riemannian manifold, and $K$  a compact, connected Lie group  acting isometrically  on $M$. For $\K=K\times K$, consider the oscillatory integral
\begin{align*}
I(\mu)
&=  \int_{\K} \int _{T^\ast \widetilde W}   e^{i  \Phi(p , \xi,k_1,k)/\mu }   a( (k_1,k) \cdot p,  p , \xi,k_1,k) \d(T^\ast M)( p,  \xi) d_{\K}(k_1,k), \qquad \mu >0,  
\end{align*}
where $(\phi,\widetilde W)$ are local coordinates on $M$, while $a \in \CT(\widetilde W \times T^\ast  \widetilde W\times \K)$ is an amplitude  which might depend on the parameter $\mu$, and
$
\Phi(p, \xi, k_1,k) =(\phi((k_1,k) \cdot  p) - \phi ( p)) \cdot  \xi
$.  
Furthermore, assume that   for all multi-indices one has  $ |\gd_x^\alpha \gd ^\beta_\xi \gd _{k_1}^{\delta_1} \gd _k^\delta a|  \leq C
$  with a constant $C>0$ independent of $\mu$. Then $I(\mu)$ has the asymptotic expansion 
\bqn 
I(\mu) = (2\pi \mu)^{\kappa} \mathcal{L}_0 + O\big (\mu^{\kappa+1}(\log \mu^{-1})^{\Lambda-1}\big ),  \qquad \mu \to 0^+,
\eqn
where  $\kappa$ is the dimension of a $\K$-orbit of principal type in $M$, $\Lambda$ the maximal number of elements of a totally ordered subset of the set of $\K$-isotropy types, and the leading coefficient is given by 
\bq
\label{eq:L0}
\mathcal{L}_0=\int_{\mathrm{Reg}\, \mathcal{C}} \frac { a( k_1k  \cdot p,  p , \xi,k_1,k) }{|\det   \, \Phi''(p, \xi,k_1,k)_{N_{(p, \xi, k_1,k)}\mathrm{Reg}\, \mathcal{C}}|^{1/2}} \d(\mathrm{Reg}\, \mathcal{C})(p,\xi,k_1,k),
\eq
where $\mathrm{Reg}\, \mathcal{C}$ denotes the regular part of  $\mathcal{C}$, and $\d(\mathrm{Reg}\, \mathcal{C})$ the induced volume density. 
In particular, the integral over $\mathrm{Reg}\, \mathcal{C}$ exists.
\end{theorem}
\begin{proof}
See \cite{ramacher10}, Theorem 11.
\end{proof}

As a consequence, one obtains the following asymptotic description  as $t \to 0$ for $\tr   \pi(H^\sigma_{f_t})=\tr (P_\sigma \circ \pi(f_t)\circ P_\sigma)$. 

\begin{theorem}
\label{thm:B}
Let $\sigma \in \widehat{K}$, and $t>0$. Then 
\begin{align*}
\tr   \pi(H^\sigma_{f_t})&=\frac{d_{\sigma \otimes \sigma} }{(2\pi)^{n-\kappa}t^{(n-\kappa)/q}}\sum_\iota \int_{\Reg \mathcal{C}}\alpha_\iota ( p)    b_{f_t}^{ \iota}(\phi_\iota(p),\xi/t^{1/q};k_1,k)  \overline{(\chi_\sigma\otimes \chi_\sigma)(k_1,k)} j_\iota(p)\\ & \cdot \frac {d(\Reg \mathcal{C})(p,\xi,k_1,k)}{|\det   \, \Phi_{\iota \iota} ''(p, \xi,k_1,k)_{N_{(p, \xi, k_1,k)}\mathrm{Reg}\, \mathcal{C}}|^{1/2}}+ O(t^{-(n-\kappa-1)/q} (\log t)^{\Lambda -1}),
\end{align*}
where  $\kappa$ is the dimension of a $K$-orbit of principal type in $M$, and $\Lambda$ the maximal number of elements of a totally ordered subset of the set of $\K$-isotropy types. 
\end{theorem}
\begin{proof}
This  is an immediate consequence of Corollary \ref{cor:3}, and Theorem \ref{thm:I(mu)}, together with Lebesgue's theorem on bounded convergence.
\end{proof}
In general, it is not possible to obtain more explicit expressions for the leading term, unless one has more knowledge about the Langlands kernels $f_t$ as $t \to 0$. In particular, the bounds \eqref{eq:boundlang} are not sufficient for this purpose. We shall therefore make the following assumption, which should  hold in many cases.

\medskip

\noindent
{\bf Assumption 1.} The function  $f_t$ has  an asymptotic expansion of the form
\bqn 
f_t(g)\sim \frac{1}{ t^{d/q}} e^{-b\big (\frac{|g|^q}t\big )^{1/(q-1)}} \sum_{j=0}^\infty c_j(g) t^j, \qquad  |g| <<1,
\eqn
where $b>0$, and the coefficients $c_j(g)$ are analytic in $g$. 

\medskip

We then have the following
\begin{corollary} 
\label{cor:B}
Let Assumption 1 be fulfilled. Then 
\begin{align*}
\tr   \pi(H^\sigma_{f_t})&=\frac{d_{\sigma\otimes \sigma} }{(2\pi)^{n-\kappa}t^{(n-\kappa)/q}} [(\pi_\sigma\otimes \pi_\sigma)_{|\mathbb{H}}:\1] \sum_\iota   \int_{{\mathrm{Reg}} \, \Xi}\hat \F_\iota (p, \xi) \alpha_\iota(p)  j_\iota(p)  \frac {  d(\Reg \Xi)(p,\xi)}{\vol \mathcal{O}_{(p,\xi)}} \\&+ O(t^{-(n-\kappa-1)/q} (\log t)^{\Lambda -1}),
\end{align*}
where $\hat \F_\iota (p, \xi)=  c_0(e)\int_{\R^d} e^{i \sum_{l,j} c^j_l(p) \zeta_l \xi_j}  e^{- b|e^{ \sum \zeta_i X_i}|^{q/(q-1)}}  \Psi^\ast(d_G)(\zeta)$ is rapidly falling in $\xi$, and $\mathcal{O}_{(p,\xi)}$ denotes the $\K$-orbit in $T^\ast M$ through $(p,\xi)$, while $[(\pi_\sigma\otimes \pi_\sigma)_{|\mathbb{H}}:\1]$ is the multiplicity of the trivial representation in the restriction  of the unitary irreducible representation $\pi_\sigma\otimes \pi_\sigma$ to a principal isotopy group $\mathbb{H}\subset \K$. Actually,
\bqn 
\widetilde{\vol}(\Xi/\K)=\sum_\iota \int_{{\mathrm{Reg}} \, \Xi}\hat \F_\iota (p, \xi) \alpha_\iota(p)  j_\iota(p)  \frac {  d(\Reg \Xi)(p,\xi)}{\vol \mathcal{O}_{(p,\xi)}}
\eqn
represents a Gaussian volume of the symplectic quotient $\Xi/\K$.

\end{corollary}
\begin{proof}
On $\Reg \mathcal{C}$ we have $k_1 k \cdot p =p$, so that 
\begin{align*}
b_{f_t}^{\iota}(\phi_\iota(p),\xi;k_1,k)&=e^{-i\phi_{ \iota}(p )\cdot \xi}   \int_{U}e^{i\phi_{\iota}(k_1gk_1^{-1}  \cdot p)\cdot\xi} \alpha'_{\iota}(k_1gk_1^{-1} \cdot p)f_t(g)\beta(g) \d g \\
&=   \int_{U}e^{i[\phi_{\iota}(g  \cdot p)-\phi_{ \iota}(p )] \cdot\xi} \alpha'_{\iota}(g \cdot p) (f_t\beta)(k_1^{-1}gk_1) \d g,
\end{align*}
since we can assume that $U$ is invariant under conjugation with $K$. Consider further, with respect to the coordinates \eqref{eq:cc},  the expansion
\bqn
[\phi_{\iota}(g  \cdot p)-\phi_{ \iota}(p )]_j = \sum_{|\alpha|>0} c_\alpha^j(p) \zeta^\alpha, \qquad g \in U, \, p  \in \widetilde W_\iota,
\eqn
where the coefficients $c_\alpha^j(p)$ depend analytically on $p$. Under Assumption 1, Taylor expansion in $\tau=t^{1/q}$ at $\tau=0$ gives
\begin{align}
\begin{split}
\label{eq:1.2.2012}
b_{f_t}^{ \iota}& (\phi_\iota(p),\xi/t^{1/q};k_1,k)=   \int_{t^{-1/q} \Psi^{-1}(U)}  e^{i \sum_{|\alpha|>0, j} c^j_\alpha(p) (t^{1/q} \zeta)^\alpha \xi_j/ t^{1/q}}  \alpha'_\iota(  e^{t^{1/q} \sum \zeta_i X_i} \cdot p) \\ 
& \cdot  t^{d/q}(f_t\beta)(k_1^{-1}e^{t^{1/q} \sum \zeta_i X_i}k_1) \Psi^\ast(d_G)(\zeta) \\
&=   \int_{t^{-1/q} \Psi^{-1}(U)}  e^{i \sum_{|\alpha|>0, j} c^j_\alpha(p) (t^{1/q} \zeta)^\alpha \xi_j/ t^{1/q}}  \alpha'_\iota(  e^{t^{1/q} \sum \zeta_i X_i} \cdot p)  c_0\Big (e^{t^{1/q}\sum \zeta_i \Ad(k_1^{-1})X_i}\Big)\\ 
& \cdot  e^{- b(|e^{t^{1/q}\sum \zeta_i X_i}|^q/t)^{1/(q-1)}} \beta(k_1^{-1}e^{t^{1/q} \sum \zeta_i X_i}k_1) \Psi^\ast(d_G)(\zeta) +O(t) \\
&= \alpha_\iota'(p) c_0(e) \int_{\R^d} e^{i \sum_{l,j} c^j_l(p) \zeta_l \xi_j} e^{- b|e^{\sum \zeta_i X_i}|^{q/(q-1)}} \Psi^\ast(d_G)(\zeta) + O(t^{1/q}),
\end{split}
\end{align}
where the notation is the same as in the proof of Theorem \ref{thm:A}. Here we took into account that  by Proposition \ref{prop:Ad(k)inv} we have  $|g|=|kgk^{-1}|$ for all $g \in G$ and $k\in K$.
Furthermore, $|e^{t^{1/q} \sum \zeta_i X_i}|^q/t=|e^{\sum \zeta_i X_i}|$.
Let us now remark that for any  smooth, compactly supported function $u$ on $\Xi  \cap T^\ast \widetilde W_\iota$, and any $v \in \Cinft(\mathbb{K})$, one has the formula
\bq
\label{eq:4.2.2012}
\int_{{\mathrm{Reg}} \, {\mathcal{C}}}\frac{v(k_1,k) u(p,\xi)d({\mathrm{Reg}}  \, {\mathcal{C}})(p, \xi,k_1,k)}{|\det  \, \Phi_{\iota \iota}'' (p,\xi,k_1,k)_{|N_{(p, \xi,k_1,k)}{\mathrm{Reg}} \, {\mathcal{C}}} |^{1/2}} 
=  \int_{\mathbb{H}}  v(k_1,k) \d k_1 \d k  \cdot \int_{{\mathrm{Reg}} \, \Xi} u (p, \xi) \frac{d({\mathrm{Reg}}\, \Xi)(p, \xi)}{\vol \mathcal{O}_{(p,\xi)}},
\eq
 compare \cite{cassanas-ramacher09}, Lemma 7, where $\mathbb{H}$ is a principal $\K$-isotropy group, and  $\mathcal{O}_{(p,\xi)}$  the $\K$-orbit in $T^\ast M$ through $(p,\xi)$. In particular, 
 \bqn 
 \int_{\mathbb{H}} \overline {(\chi_\sigma\otimes \chi_\sigma)(k_1,k)} dk_1dk= [(\pi_\sigma\otimes \pi_\sigma)_{|\mathbb{H}}:\1],
 \eqn 
 where $[(\pi_\sigma\otimes \pi_\sigma)_{|\mathbb{H}}:\1]$ denotes the multiplicity of the trivial representation in the restriction  to $\mathbb{H}$ of the unitary irreducible representation $\pi_\sigma\otimes \pi_\sigma$. 
 The assertion now follows with Theorem \ref{thm:B}.
\end{proof}

 To motivate Assumption 1, and to  illustrate our results, let us  consider the classical heat kernel on $G$. Thus, consider a Cartan decomposition of $\g$ as in \eqref{eq:cartan}, and  let $X_1, \dots X_p$ be an orthonormal basis of $\p$, and $Y_1, \dots, Y_l$ an orthonormal basis for $\k$ with respect to $\langle \cdot ,\cdot \rangle _\theta$. If $\Omega$ and $\Omega_K$ denote the Casimir elements of $G$ and $K$, one has
 \bqn 
 \Omega= \sum _{i=1}^p X_i^2- \sum _{i=1}^l Y_i^2, \qquad \Omega_K=-\sum _{i=1}^l Y_i^2.
 \eqn
Let
\bqn 
P=-\Omega+2\Omega_K=-\sum _{i=1}^p X_i^2- \sum _{i=1}^l Y_i^2.
\eqn
Then $dR(P)$ is the Beltrami-Laplace operator $\Delta_G$ on $G$ with respect to the left invariant metric. $dR(P)$ is a strongly elliptic operator associated to $R$, and generates a strongly continuous semigroup which coincides with  the classical heat semigroup $e^{-t \Delta_G}$, whose kernel $p_t$ is given by the corresponding universal Langlands kernel. In particular, 
\bq
\label{eq:22}
e^{-t \Delta_G}= R(p_t),
\eq
see \cite{mueller98}, Section 3. Let us now recall that on Riemannian manifolds admitting a properly discontinuous group of isometries with compact quotient, a fundamental solution  of the heat equation with Gaussian bounds can be constructed explicitly \cite{donnelly79}. Furthermore, every real, semisimple Lie group possesses  a discrete, torsion-free subgroup  with compact quotient \cite{borel63}. If therefore $H(t,g,h)$ is the fundamental solution of the heat equation $\gd/\gd t+\Delta_G$ on $G$ constructed in this way, the Gaussian bounds imply that it coincides with the Langlands kernel $p_t$, so that $H(t,g,h)=p_t(g^{-1} h)$. Furthermore, one has an asymptotic expansion of the form
\bqn 
H(t,g,h) \sim (4\pi t)^{-d/2} e^{-\frac {d^2(g,h)}{4t}} \sum_{j=0}^\infty t^j u_j(g,h),
\eqn
valid in a sufficiently small neighborhood of the diagonal in $G \times G$, see  \cite{donnelly79},  Theorem 3.3. As before,  $d(g,h)$ denotes the geodesic distance between two points with respect to the left invariant metric on $G$, and $u_0(g,g)=1$. Corollary \ref{cor:B} then implies
\begin{corollary} 
\label{cor:5}
Let $\Delta_G$ be the Laplace-Beltrami operator on $G$, and $p_t\in \S(G)$ its heat kernel. Then 
\begin{align*}
\tr   \pi(H^\sigma_{p_t})&=\frac{d_{\sigma\otimes \sigma}}{(2\pi)^{n-\kappa}t^{(n-\kappa)/2}} [(\pi_\sigma\otimes \pi_\sigma)_{|\mathbb{H}}:\1]\widetilde {\vol}(\Xi/\K)  + O(t^{-(n-\kappa-1)/2} (\log t)^{\Lambda -1}),
\end{align*}
where 
\bqn 
\widetilde {\vol}(\Xi/\K) =\sum_\iota \int_{{\mathrm{Reg}} \, \Xi}\hat \F_\iota (p, \xi) \alpha_\iota(p)  j_\iota(p)  \frac {  d(\Reg \Xi)(p,\xi)}{\vol \mathcal{O}_{(p,\xi)}},
\eqn
and 
$\hat \F_\iota (p, \xi)=   (4\pi)^{-d/2} \int_{\R^d} e^{i \sum_{l,j} c^j_l(p) \zeta_l \xi_j}  e^{- |e^{ \sum \zeta_i X_i}|^{2}/4}  \Psi^\ast(d_G)(\zeta)$. 
\end{corollary}
\qed

\section{Homogeneous vector bundles on compact locally symmetric spaces}

In this section, we apply the previous analysis  to heat traces of Bochner-Laplace operators on compact, locally symmetric spaces. In the rank one case, this problem was already considered by  Miatello  \cite{miatello80} and DeGeorge and Wallach \cite{degeorge-wallach79}. As before, let $G$ denote a connected, real, semisimple Lie group with finite center, and $\Gamma$  a discrete, uniform subroup of $G$. Consider $M=\Gamma \backslash G$, and denote by
$
 \pi_\Gamma(g)\phi(h) = \phi(hg)$, $g,\, h \in G
$, 
 the right regular representation  \footnote{More precisely, 
 $ \pi_\Gamma(g)\phi(\Gamma h) = \phi(\Gamma hg)$, where $\Gamma h \in \gmg$, $g \in G$.} of $G$ in 
 the space $L^2(\Gamma \backslash G)$ of square integrable functions on $\Gamma \backslash G$.  Since $\gmg$ is compact,  the right regular representation decomposes discretely according to 
\bq
\label{eq:21}
\pi_\Gamma \simeq \bigoplus_{\rho \in \widehat{G}} m_\rho \pi_\rho,
\eq
where $\widehat{G}$ stands for the set of equivalence classes of irreducible unitary representations of $G$, $(\pi_\rho, H_\rho) \in  \rho$,  and $m_\rho <\infty$ denotes the multiplicity of $\rho$ in $(\pi_\Gamma, \L^2(\Gamma \backslash G))$. For $f \in C^\infty_c(G)$, the Bochner integral $\pi_\Gamma(f) = \int_G f(g) \pi_\Gamma (g)\d g$ defines a bounded  operator  on $\ltwo$ whose kernel is given by the   $C^\infty$ function
\bq
\label{eq:14}
k_f(g,h)= \sum_{\gamma \in \Gamma} f(g^{-1}\gamma h), \qquad g,h \in G,
\eq
the series converging uniformly on compacta. The regularity of the kernel implies that $\pi_\Gamma(f)$ is of trace class, and 
\[
\tr \pi_\Gamma(f) = \int_{\gmg}k_f(g,g) \d g = \int_{\gmg} \sum_{\gamma \in \Gamma} f(g^{-1}\gamma g) \d g.
\]
Note that we are slightly abusing of notation, and  denoting the invariant measure on $\Gamma \backslash G$ also by $dg$. 
If  $f \in L^1(G)$,  the operator $\pi_\Gamma(f)$ is still defined, but  might not be of trace class. If $f \in \S(G)$ is rapidly falling, it was shown in Theorem \ref{thm:2} that $\pi_\Gamma(f)$ is a smooth operator, which by Corollary \ref{cor:1} implies that it  has a well-defined  trace. As in the case of a compactly supported $f$, one can  show that for $f \in \S(G)$   the kernel of $\pi_\Gamma(f)$ is given globally by  the expression \eqref{eq:14}, and that it satisfies Selberg's trace formula. Indeed, one has the following

\begin{lemma}\label{prop01}
Let $f \in \S(G)$ be a rapidly decaying function on $G$.  Then the series $k_f(g,h)=\sum_{\gamma \in \Gamma}f (h^{-1}\gamma g)$ converges uniformly on compacta to a $C^\infty$ function, and represents the integral kernel of the bounded operator $\pi_\Gamma(f):\L^2(\gmg) \rightarrow \L^2(\gmg)$.
\end{lemma}

\begin{proof}
 By Definition \ref{def:1},  for all $\kappa > 0$  we have  the  inequality
\begin{equation*}
|f(h^{-1}\gamma g)| \le C_\kappa e^{-\kappa|h^{-1}\gamma g|}, \qquad g,h \in G,
\end{equation*}
as well as for all derivatives of {all} orders of $f$.  Consequently,
\begin{equation}\label{boundps1}
\sum_{\gamma \in \Gamma}|f(h^{-1}\gamma g)| \le \sum_{\gamma \in \Gamma} e^{-\kappa\, d(h^{-1}\gamma g, e)} = \sum_{\gamma \in \Gamma} e^{-\kappa\, d(\gamma g, h)},
\end{equation}
since  left-translation by $h$ is an isometry.  Now, recall that  for a metric space $({\bf X},d)$, and a discrete infinite subgroup $\Gamma' \subset \mathrm{Iso}({\bf X})$ of the isometry group of ${\bf X}$ the corresponding Poincar\'e series  is defined by
\begin{equation}\label{poincser}
P(s,p, q)= \sum_{\gamma \in \Gamma'} e^{-s\, d(p,\gamma q)}, \qquad p,q \in {\bf X}, \quad s >0.
\end{equation}
By general theory \cite{nichols89}, for each discrete subgroup $\Gamma'$, there exists a $\delta_{\Gamma'} >0$, called the critical exponent of $\Gamma'$, such that $P(s,p,q)$ converges for $s> \delta_{\Gamma'}$ and diverges for $s < \delta_{\Gamma'}$. Furthermore, the exponent $\delta_{\Gamma'}$ does not depend on $p$ or $q$. 
The estimate \eqref{boundps1}   means that for fixed $g,h \in G$,  the series $k_f(g,h)$ is  majorized by the Poincar\'e series $\sum_{\gamma \in \Gamma} e^{-\kappa\, d(\gamma g, h)}$. Choosing  $\kappa > \delta_\Gamma$, 
we deduce  that  $k_f(g,h)$ is absolutely convergent for fixed $g,h \in G$. To see that $(g,h) \mapsto k_f(g,h)$ is continuous,  note that 
\begin{align*}
\Big|k_f(h,g) - k_f(z,g)\Big| &= \Big| \sum_{\gamma \in \Gamma}f(h^{-1}\gamma g) - \sum_{\gamma \in \Gamma}f(z^{-1}\gamma g)\Big| \\ & \le \sum_{\substack{\gamma \in \Gamma\\ \|\gamma\| \le R}}\Big|f(h^{-1}\gamma g) - f(z^{-1}\gamma g)\Big| + \Big| \sum_{\substack{\gamma \in \Gamma\\ \|\gamma\| > R}} f(h^{-1}\gamma g)\Big| + \Big| \sum_{\substack{\gamma \in \Gamma\\ \|\gamma\| > R}} f(z^{-1}\gamma g)\Big|.
\end{align*}
Since the sum $\sum_{\gamma \in \Gamma}f(h^{-1}\gamma g)$ converges, the last two terms in the last inequality can be made as small as required by choosing $R$ big enough,  while  the first term becomes   small if $d(h,z)$ is small, being a finite some of continuous functions.  Thus, $(g,h) \mapsto k_f(g,h)$ is continuous. Since the same argument works for all  derivatives,  $f(h,g)$ converges uniformly on compacta to a $C^\infty$ function.
To see that $k_f(g,h)$ represents the Schwartz kernel of $\pi_\Gamma(f)$ for $f \in \S(G)$, note that $\pi_\Gamma(f)$ acts on $\phi \in L^2(\Gamma \backslash G)$ according to
\begin{equation}\label{tff1}
(\pi_\Gamma(f)\phi)(h) =\int_G f(g)\phi(hg)\d g = \int_G f(h^{-1}g)\phi(g)\d g, \qquad h \in G,
\end{equation}
the integral being absolutely convergent due to the inequality
\[
|\alpha \beta| \le \frac{1}{2}|\alpha|^2 + \frac{1}{2}|\beta|^2, \quad \alpha, \, \beta \in \C,
\]
and the  fact that if $f$ is rapidly decreasing,  $f\cdot\overline{f}$ is rapidly decreasing, too. By   Fubini's theorem, and the first part of the  lemma  we therefore obtain for each $h \in G$
\begin{align*}\label{tff1}
(\pi_\Gamma(f)\phi)(h) &= \int_{\Gamma \backslash G}  \left(\sum_{\gamma \in \Gamma}f(h^{-1}\gamma g) \phi(\gamma g)\right) \d g =\int_{\Gamma \backslash G}  k_f(h, g) \phi(g) \d g,
\end{align*}
since $\phi(\gamma g) =\phi(g)$. Thus, $\pi_\Gamma(f)$ is an integral operator with kernel $k_f(h,g)\in \Cinft(\gmg\times \gmg)$. 
\end{proof}

\begin{corollary}
Let $f \in \S(G)$. Then $f$ satisfies Selberg's trace formula
 \bq
\label{eq:selberg}
\bigoplus_{\rho \in \widehat{G}} m_\rho \tr  \pi_\rho (f) =\sum_{[\gamma]}\vol(\Gamma_\gamma \backslash G_\gamma) \int_{G_\gamma \backslash G} f(g^{-1}\gamma g) \d g,
\eq
where $[\gamma]$ denotes the conjugacy class of $\gamma$ in $  \Gamma$, and $\Gamma_\gamma$ and $G_\gamma$ are the centralizers of $\gamma$ in $\Gamma$ and $G$, respectively.
 \end{corollary}
 \begin{proof}
  Lemma \ref{Mia} and \ref{prop01}  yield 
\[
\tr \pi_\Gamma(f) = \int_{\Gamma \backslash G} k_f(g,g)dx = \int_{\Gamma \backslash G}\sum_{\gamma \in \Gamma}f(g^{-1}\gamma g)\d g.
\]
 Denoting by  $[\gamma]$  the conjugacy class of $\gamma$, and by $\Gamma_\gamma$ the centralizer of $\gamma$ in $\Gamma$ one deduces
\[
\tr \pi_\Gamma(f)= \int_{\Gamma \backslash G}\sum_{[\gamma]}\sum_{\delta \in \Gamma_\gamma \backslash \Gamma}f(g^{-1}\delta^{-1}\gamma\delta g)\d g = \sum_{[\gamma]}\int_{\Gamma \backslash G}\sum_{\delta \in \Gamma_\gamma \backslash \Gamma}f(g^{-1}\delta^{-1}\gamma\delta g)\d g,
\]
everything being uniformly convergent. Replacing the inner sum by an integral with a counting measure $d\delta$ yields 
\[
\tr \pi_\Gamma(f)= \sum_{[\gamma]}\int_{\Gamma \backslash G}\int_{\Gamma_\gamma \backslash \Gamma}f(g^{-1}\delta^{-1}\gamma\delta g)d\delta \d g
= \sum_{[\gamma]}\int_{\Gamma_\gamma \backslash G} f(y^{-1}\gamma y)dy,
\]
where we took into account  that for any sequence $G_1\subset G_2 \subset G$ of unimodular groups, a right invariant measure on $G_1 \backslash G$  can be written as the product of right invariant measures on $G_2\backslash G$, and $G_1\backslash G_2$,  respectively. With the same argument, the above equality can be rewritten as
\[
\tr \pi_\Gamma(f)= \sum_{[\gamma]}\int_{G_\gamma \backslash G}\int_{\Gamma_\gamma \backslash G_\gamma}f(v^{-1}u^{-1}\gamma uv)du dv,
\]
where  $G_\gamma$ denotes the centralizer of $\gamma$ in $G$. Since $u^{-1} \gamma u=\gamma$, and $\Gamma_\gamma \backslash G_\gamma$ is compact, one finally obtains the {geometric side} of the trace formula
\[
\tr \pi_\Gamma(f) = \sum_{[\gamma]}\text{vol}(\Gamma_\gamma \backslash G_\gamma) \int_{G_\gamma \backslash G} f(g^{-1}\gamma g) \d g.
\]
 To obtain the spectral side, note that according to the decomposition \eqref{eq:21} we have
\[
\tr \pi_\Gamma(f) =\bigoplus_{\rho \in \widehat{G}} m_\rho \tr \pi_\rho(f),
\]
 where $\pi_\rho(f) = \int_G f(g) \pi_\rho(g)\d g$ is of trace class,  and defines a distribution 
\bqn 
\theta_\rho: \CT(G) \ni f \mapsto  \tr\pi_\rho(f)\in \C
\eqn
on $G$ which represents  the global character of $\rho$. 
 Selberg's trace formula for $f \in \S(G)$ now follows.
\end{proof}

Consider next  a maximal compact subgroup $K$ of $G$, and $\sigma \in \widehat{K}$. As a consequence of Theorem \ref{thm:B}, and Selberg's formula \eqref{eq:selberg} we obtain
\begin{proposition}
\label{prop:2}
Let  $f_t\in \S(G)$, $t >0$, be the Langlands kernel of a semigroup generated by a strongly elliptic operator associated to the representation $\pi_\Gamma$. Then
\begin{gather*}
   (L_\sigma f_t)(e) =\frac{d_{\sigma \otimes \sigma} }{(2\pi)^{\dim G/K}\vol (\gmg) \, t^{\frac{\dim G/K}{q}}}\sum_\iota \int_{\Reg \mathcal{C}}\alpha_\iota ( p)    b_{f_t}^{ \iota}(\phi_\iota(p),\xi/t^{1/q};k_1,k)  \overline{(\chi_\sigma\otimes \chi_\sigma)(k_1,k)} \\  \cdot j_\iota(p) \frac {d(\Reg \mathcal{C})(p,\xi,k_1,k)}{|\det   \, \Phi_{\iota \iota}''(p, \xi,k_1,k)_{N_{(p, \xi, k_1,k)}\mathrm{Reg}\, \mathcal{C}}|^{1/2}},
\end{gather*}
up to terms of order $ O(t^{-(\dim G/K-1)/q} (\log t)^{\Lambda -1})$, the notation being as in Theorem \ref{thm:B}. Here $L_\sigma$ denotes the projector onto the isotypic component $\L^2(G)_\sigma$. If, in addition,  Assumption 1 is satisfied, the leading term of $(L_\sigma f_t)(e)$ is given by
\begin{align*}
\frac{d_{\sigma \otimes \sigma}  [(\pi_\sigma\otimes \pi_\sigma)_{|\mathbb{H}}:\1]}{(2\pi)^{\dim G/K}\vol (\gmg) \, t^{\frac{\dim G/K}{q}}} \widetilde {\vol}(\Xi/\K), 
\end{align*}
where 
\bqn
\widetilde {\vol}(\Xi/\K) =\sum_\iota \int_{{\mathrm{Reg}} \, \Xi}\hat \F_\iota (p, \xi) \alpha_\iota(p)  j_\iota(p)  \frac {  d(\Reg \Xi)(p,\xi)}{\vol \mathcal{O}_{(p,\xi)}},
\eqn
and $\hat \F_\iota (p, \xi)=  c_0(e)\int_{\R^d} e^{i \sum_{l,j} c^j_l(p) \zeta_l \xi_j}  e^{- b|e^{ \sum \zeta_i X_i}|^{q/(q-1)}}  \Psi^\ast(d_G)(\zeta)$.
\end{proposition}
\begin{proof} 
By Theorem \ref{thm:A}, $\tr \pi_\Gamma(H^\sigma_{f_t\beta})=\tr \pi_\Gamma(H^\sigma_{f_t})+O(t^\infty)$, where $0 \leq \beta \leq 1$ is a test function on $G$ with support in a sufficiently small neighborhood $U$ of $e \in G$, and which is equal $1$ close to $e$. Furthermore, by Theorem \ref{thm:B}, 
\begin{gather*}
\tr \pi_\Gamma(H^\sigma_{f_t})=\frac{d_{\sigma \otimes \sigma} }{(2\pi)^{\dim G/K}t^{\frac{\dim G/K}{q}}}\sum_\iota \int_{\Reg \mathcal{C}}\alpha_\iota ( p)    b_{f_t}^{ \iota}(\phi_\iota(p),\xi/t^{1/q};k_1,k)  \overline{(\chi_\sigma\otimes \chi_\sigma)(k_1,k)} j_\iota(p)\\  \cdot \frac {d(\Reg \mathcal{C})(p,\xi,k_1,k)}{|\det   \, \Phi_{\iota \iota}''(p, \xi,k_1,k)_{N_{(p, \xi, k_1,k)}\mathrm{Reg}\, \mathcal{C}}|^{1/2}}+ O(t^{-(\dim G/K-1)/q} (\log t)^{\Lambda -1}).
\end{gather*}
Next, recall that for any $\gamma \in \Gamma \subset G$, the $G$-conjugacy class $[\gamma]_G$ is closed. Furthermore,  every  compactum in G meets only finitely many $[\gamma]_G$, see \cite{mostow70}, Lemma 8.1. Consequently, by choosing the support of $\beta$ sufficiently small, we obtain with \eqref{eq:selberg} 
\bqn 
\tr \pi_\Gamma(H^\sigma_{f_t\beta})= \vol (\gmg) \, H^\sigma_{f_t}(e),
\eqn
and the assertion follows with Theorem \ref{thm:B}, and Corollary \ref{cor:B}. 
 \end{proof}

We  now apply our results to heat kernels of Bochner-Laplace operators on compact, locally symmetric spaces. Let $(\pi_\sigma,V_\sigma)$ be an irreducible unitary representation of $K$ of class $\sigma \in \widehat{K}$. Consider the associated homogeneous vector bundle $\widetilde E_\sigma=(G \times V_\sigma)/K$ over $G/K$, and endow it with the $G$-invariant Hermitian fibre metric  induced by the inner product in $V_\sigma$. Let $\g=\k\oplus \p$ be a Cartan decomposition of $\g$ as in \eqref{eq:cartan}, and consider the unique $G$-invariant connection $\widetilde \nabla$ on $\widetilde E_\sigma$ given by the condition that if $s$ is a smooth cross section, $Y \in \p$, and $\Pi: G \rightarrow G/K$ is the canonical projection, then
\bqn 
\widetilde \nabla_{\Pi_\ast(Y)}(s)= \frac{d}{dt} s( \e{tY} K)_{|s=0},
\eqn
$\Pi_\ast$ being the differential of $\Pi$ at $e \in G$. Let further $\widetilde \Delta_\sigma= \widetilde \nabla^\ast \widetilde \nabla$ be the Bochner Laplace operator of $\widetilde \nabla$, and denote by $\Cinft(\widetilde E_\sigma)$, $\CT(\widetilde E_\sigma)$, and $\L^2(\widetilde E_\sigma)$ the usual spaces of sections of $\widetilde E_\sigma$. With respect to the identification
\bqn 
\Cinft(\widetilde E_\sigma)=(\Cinft(G) \otimes V_\sigma)^K, \eqn
where $(\Cinft(G) \otimes V_\sigma)^K=\mklm{\phi:G\rightarrow V_\sigma \text{ is smooth  and }  \phi(gk ) = \pi_\sigma(k)^{-1} \phi(g), \, k \in K, \, g \in G}$, and the corresponding identifications for $\CT(\widetilde E_\sigma)$ and $\L^2(\widetilde E_\sigma)$, one has
\bqn 
\widetilde \Delta_\sigma=-dR(\Omega) \otimes \id + \id \otimes d\pi_\sigma (\Omega_K)=-dR(\Omega) \otimes \id + \lambda_\sigma \id
\eqn
for some $\lambda_\sigma \geq 0$, $\Omega$ and $\Omega_K$ being the Casimir elements of $G$ and $K$, respectively, see \cite{miatello80}, Proposition 1.1. As it turns out, the operator $\widetilde \Delta_\sigma: \CT(\widetilde E_\sigma) \rightarrow \L^2(\widetilde E_\sigma)$ is essentially self-adjoint, and has a unique self-adjoint extension which we shall also denote by $\widetilde \Delta_\sigma $. It is a positive operator, and we denote  the corresponding heat semigroup by $e^{-t\widetilde \Delta_\sigma}$. It is given by
\bq
\label{eq:28}
(e^{-t\widetilde \Delta_\sigma} \phi)(g)= \int_G h_t^\sigma ( g_1) \phi(g g_1) \d g_1, \quad \phi \in (\L^2(G) \otimes V_\sigma)^K, 
\eq
where $h_t^\sigma : G \rightarrow \mathrm{End } (V_\sigma)$ is square integrable, and has the covariance property
\bqn 
h_t^\sigma (g)= \pi_\sigma(k)h_t^\sigma (k^{-1} g k_1) \pi_\sigma(k_1)^{-1}, \qquad g \in G, \, k, k_1 \in K. 
\eqn
As one can show, $h_t^\sigma$ is actually given in terms of the classical heat kernel $p_t$ introduced in \eqref{eq:22} according to 
\bq
\label{eq:29}
h_t^\sigma (g)= e^{t \lambda_\sigma} \int_K \int_K p_t(k^{-1} g k_1) \pi_\sigma(k k_1^{-1}) \d k_1 \d k,
\eq
see \cite{barbasch-moscovici83} and \cite{mueller98}, Section 3. Let now $\Gamma $ be a discrete, uniform, torsion-free subgroup of $G$. Then $\Gamma$ acts without fixed points on $G/K$,  and $\Gamma \backslash G/K$ constitutes a compact, locally symmetric space. Let $E_\sigma= \Gamma \backslash \widetilde E_\sigma \rightarrow \Gamma \backslash G/K $ be the pushdown of the homogenous vector bundle $\widetilde E_\sigma\rightarrow G/K$. Again, we have identification
\bqn 
\Cinft( E_\sigma)=(\Cinft(\gmg) \otimes V_\sigma)^K,
\eqn
and similarly for $\CT(E_\sigma)$, and $\L^2(E_\sigma)$. 
Since $\widetilde \Delta_\sigma$ is $G$-invariant, it induces an elliptic, essentially self-adjoint operator $ \Delta_\sigma=\nabla^\ast  \nabla: \Cinft( E_\sigma) \rightarrow \L^2( E_\sigma)$, where $\nabla$ is the pushdown of the canonical connection $\widetilde \nabla$. Let $e^{-t \Delta_\sigma}$ be the corresponding heat semigroup. With respect to a basis $\mklm{e_i}$ of $V_\sigma$, we obtain with \eqref{eq:28} and \eqref{eq:29}  
\bqn 
[e^{-t \Delta_\sigma} \phi)(g)]_j= \sum_{k=1}^{\dim \sigma} \pi_\Gamma( ^{jk} H_t^\sigma) [\phi(g)]_k,   \quad \phi \in (\L^2(\Gamma \backslash G) \otimes V_\sigma)^K, 
\eqn
where 
\bqn 
^{jk} H_t^\sigma(g) =e^{t \lambda_\sigma} \int_K \int_K p_t(k^{-1} g k_1) (\pi_\sigma(k k_1^{-1}))_{jk} \d k_1 \d k.
\eqn
Thus, $e^{-t \Delta_\sigma}$ is given by the matrix of convolution operators $\pi_\Gamma( ^{jk} H_t^\sigma)$. The kernels $ ^{jk} H_t^\sigma$ are essentially of the  same form as the kernels $H_{p_t}^\sigma$ defined in \eqref{eq:11a}, and we arrive at
\begin{theorem}
\label{thm:5}
Let $\sigma \in \widehat{K}$, and $\Delta_\sigma$ be the Bochner-Laplace operator on the homogeneous vector bundle $E_\sigma =\Gamma \backslash (G\times V_\sigma)/K\rightarrow \gmg/K$. Then 
\begin{gather*}
\tr e^{-t \Delta_\sigma} =\frac{ e^{t\lambda_\sigma} \int_{\mathbb{H}} \tr \pi_\sigma(kk_1^{-1}) dk_1 dk}{(2\pi)^{\dim G/K}t^{\frac{\dim G/K}{2}}} \widetilde {\vol}(\Xi/\K) + O(e^{t \lambda_\sigma}t^{-(\dim G/K-1)/2} (\log t)^{\Lambda -1}) ,
\end{gather*}
  where 
\bqn 
\widetilde {\vol}(\Xi/\K)=\sum_\iota   \int_{\Reg \Xi} \hat \F_\iota (p, \xi)  \frac { \alpha_\iota(p)  j_\iota(p)  d(\Reg \Xi)(p,\xi)}{\vol \, \mathcal{O}_{(p,\xi)}},
\eqn
and   
$
\hat \F_\iota (p,\xi)= (4\pi)^{-d/2} \int_{\R^d} e^{i \sum_{l,j} c^j_l(p) \zeta_l \xi_j}  e^{-  |e^{\sum \zeta_iX_i}|^2/4} \Psi^\ast(d_G)(\zeta), \, b>0,
$
the notation being the same as in Corollary \ref{cor:5}. 
\end{theorem}
\begin{proof}
By Theorem \ref{thm:B}, and \eqref{eq:1.2.2012}, we have
\begin{gather*}
\tr \pi_\Gamma( ^{jk} H_t^\sigma)  =\frac{ e^{t \lambda_\sigma}}{(2\pi)^{\dim G/ K}t^{\frac{\dim G/K}{2}}}\sum_\iota   \int_{\Reg \mathcal{C}} \hat \F_\iota(p,\xi)  \frac {(\pi_\sigma(k k_1^{-1}))_{jk}\alpha_\iota(p)  j_\iota(p)  d(\Reg \mathcal{C})(p,\xi,k_1,k)}{|\det   \, \Phi_{\iota \iota}''(p, \xi,k_1,k)_{N_{(p, \xi, k_1,k)}\mathrm{Reg}\, \mathcal{C}}|^{1/2}}
\end{gather*}
up to terms of order $ O(e^{t \lambda_\sigma}t^{-(\dim G/K-1)/2} (\log t)^{\Lambda -1})$. The assertion now follows with \eqref{eq:4.2.2012}.
\end{proof}


\providecommand{\bysame}{\leavevmode\hbox to3em{\hrulefill}\thinspace}
\providecommand{\MR}{\relax\ifhmode\unskip\space\fi MR }
\providecommand{\MRhref}[2]{%
  \href{http://www.ams.org/mathscinet-getitem?mr=#1}{#2}
}
\providecommand{\href}[2]{#2}


\end{document}